\documentclass{article}
\usepackage{graphicx}
\usepackage{amsmath}
\usepackage{amsfonts}
\usepackage{amssymb}
\usepackage{xcolor}
\usepackage{tikz}
\usepackage{pgfplots}
\pgfplotsset{compat=1.17}
\usepackage{hyperref}
\usepackage{cite}
\usepackage{float} 
\usepackage{amsthm}
\usepackage[a4paper, total={178mm,229mm}, top=35mm, left=19mm, right=19mm, bottom=25mm]{geometry}
\theoremstyle{definition}

\newtheorem{remark}{Remark}
\newtheorem{proposition}{Proposition}
\newtheorem{theorem}{Theorem}
\usepackage{algorithm}
\usepackage{algpseudocode}
\usepackage{subcaption} 
\usepackage{authblk}
\usepackage{comment}

\title{RUNNs: Ritz--Uzawa Neural Networks \\ for Solving Variational Problems}

\author[a,d]{Pablo Herrera\thanks{Corresponding author. E-mail address: pherrera@bcamath.org (P. Herrera).}}
\author[g]{Jamie M. Taylor}
\author[a,b]{Carlos Uriarte}
\author[d]{Ignacio Muga}
\author[c,a,f]{David Pardo}
\author[e]{Kristoffer G. van der Zee}

\affil[a]{Basque Center for Applied Mathematics, Bilbao, Spain}
\affil[b]{Curtin University, Perth, Australia}
\affil[c]{Department of Mathematics, University of the Basque Country UPV/EHU, Leioa, Spain}
\affil[d]{Instituto de Matemáticas, Pontificia Universidad Católica de Valparaíso, Valparaíso, Chile}
\affil[e]{School of Mathematical Sciences, University of Nottingham, Nottingham, UK}
\affil[f]{Ikerbasque: Basque Foundation for Science, Bilbao, Spain}
\affil[g]{Department of Mathematics, CUNEF Universidad, Madrid, Spain}

\date{\today}

\begin{document}

\maketitle

\begin{abstract}
Solving partial differential equations (PDEs) using neural networks presents different challenges, including integration errors and spectral bias, often leading to poor approximations. In addition, standard neural network-based methods, such as physics-informed neural networks (PINNs), often fail when dealing with PDEs characterized by low-regularity solutions, which are incompatible with the strong formulation given by PINNs.

To address these limitations, we introduce the Ritz--Uzawa neural networks (RUNNs) framework, an iterative methodology to solve variational problems, such as strong, weak, and ultra-weak formulations. Rewriting the PDE as a sequence of Ritz-type minimization problems within a Uzawa framework provides an iterative method that, in specific cases, reduces variance of the numerical integration error during training. We demonstrate that the strong formulation offers a passive variance reduction mechanism, whereas variance remains persistent in weak and ultra-weak formulations. Furthermore, we address the spectral bias of standard architectures through a data-driven frequency tuning strategy. By initializing a sinusoidal Fourier feature mapping based on the normalized cumulative power spectral density (NCPSD) of previous residuals or their proxies, the network dynamically adapts its frequency modes to capture high-frequency components and severe singularities. Numerical experiments demonstrate that RUNNs accurately solve highly oscillatory solutions and successfully recover a discontinuous $L^2$ solution from a distributional $H^{-2}$ source -- a scenario that is incompatible with an $H^1$-formulation.

\vspace{1em}
\end{abstract}

\section{Introduction}

Neural networks (NNs) have become useful tools for solving Partial Differential Equations (PDEs), providing a way to solve high-dimensional or parametric problems by transforming the PDE into a minimization problem via a loss function \cite{Raissi2019PINNs, EYu2018DeepRitz}. Specifically, developing robust, mesh-free solvers for differential equations allows for the accurate simulation of physical phenomena in engineering and science, avoiding the limitations of traditional mesh generation (i.e., the curse of dimensionality) \cite{han2018solving}.

Common examples of this approach include physics-informed neural networks (PINNs) \cite{Raissi2019PINNs}, which minimize the strong form of the residual, and the deep Ritz method (DRM) \cite{EYu2018DeepRitz}, which minimizes the energy functional associated with the weak variational form. Despite their use, practical challenges remain. First, there is often a mismatch between the regularity of the exact solution and the space spanned by the neural network, especially in problems with low-regularity solutions where standard PINNs may fail. Second, approximating the loss function via quadrature rules might introduce bias and variance. As shown in \cite{Rivera2022, taylor2025stochastic}, these integration errors can prevent the network from finding the true minimizer, leading to overfitting even with an ideal optimization path.

To properly accommodate problems with low-regularity solutions, where the residual naturally maps into a dual space, some approaches measure the loss in dual norms. For example, robust variational PINNs (RVPINNs) \cite{rojas2024robust} and machine-learning minimal-residual (ML-MRes) frameworks \cite{brevis2021machine} minimize residuals in dual norms. Similarly, using the classical resolution of saddle-point problems \cite{uzawa1958iterative, bacuta2006unified}, recent works have proposed neural implementations of the Uzawa iteration. These include deep Uzawa \cite{makridakis2024deep} and the Inexact Uzawa-double deep Ritz method \cite{BennyChackoBEV2025}, which replace trial and test spaces with neural networks. While these methods improve stability, they typically overlook the spectral bias inherent in standard architectures and do not explicitly address the relationship between the iterative scheme and the possible reduction of the variance in the numerical integration.

Our main contribution in this paper is the proposal of the Ritz--Uzawa neural networks (RUNNs) framework to iteratively solve linear PDEs in strong, weak, and ultra-weak variational formulations. RUNNs reformulate the PDE as a sequence of Ritz-type minimizations inside a Uzawa scheme. This framework is based on the Inexact Uzawa method \cite{bramble1997analysis}, which preserves convergence even when inner sub-problems are solved approximately. Furthermore, we explicitly address the spectral bias of neural networks \cite{tancik2020fourier,sitzmann2020implicit, wang2024} by employing Sinusoidal Fourier Feature Mapping to approximate dual variables, allowing efficient capture of high-frequency error components through iterative corrections. We also analyze the passive variance reduction property of different Uzawa formulations, showing it holds with strong formulations, but not in weak or ultra-weak ones.

The methodology approximates the primal and dual variables of the Uzawa scheme with neural networks trained via a hybrid optimization strategy combining least squares and Adam (LS/Adam) \cite{uriarte2025vpinnsLS, Cyr2019RobustTA, kingma2017adam}. This combination handles linear and non-linear parameters efficiently. We demonstrate that RUNNs show better stability compared to the standard deep Ritz method, particularly when dealing with low-regularity solutions (e.g., solutions in $L^2$ that do not belong to $H^1_0$) and high frequency behavior.

The paper is organized as follows. Section~\ref{sec:Math_framework} establishes the abstract mathematical framework and proves the convergence of the inexact Uzawa method, categorizing the resolution into three distinct optimization approaches (Approaches 1, 2, and 3). Specifically, it extends the iterative scheme to solve strong, weak, and ultra-weak variational formulations, analyzing the contraction properties of the involved operator and demonstrating how these formulations generalize existing methods like the deep double Ritz method (D$^2$RM) \cite{uriarte2023DoubleRitz} and multilevel neural networks \cite{aldirany2024multi, wang2024}. Section~\ref{sec:architecture_and_training} describes the neural network architectures employed, introducing a hybrid model that combines standard multilayer perceptrons (MLPs) with a sinusoidal Fourier feature mapping. This section details an adaptive strategy that uses the normalized cumulative power spectral density (NCPSD) of the residuals to tune the frequency bandwidth, effectively mitigating spectral bias. Furthermore, it presents the hybrid least squares/Adam (LS/Adam) optimization strategy, which decouples the training of linear and non-linear parameters using Tikhonov regularization for the output layer. Section~\ref{sec:Loss_Discretization} discusses the discretization of the loss function, implementing unbiased third-order stratified stochastic quadrature rules to control integration bias. It also provides a theoretical analysis of the variance asymptotics, highlighting the difference in passive variance reduction between strong formulations (where variance tends to zero) and variational formulations (where variance remains non-zero). Finally, Section~\ref{section:Numerical_experiments} presents numerical experiments validating the approach. These include a comparison of convergence rates between pure Adam and LS/Adam, a high-frequency problem solved via spectral matching, and a challenging test case for an ultra-weak formulation with a distributional $H^{-2}$ source term (the derivative of a Dirac delta), where the method successfully recovers a discontinuous solution in $L^2$ that falls outside the standard $H^1_0$ space.

\section{Mathematical framework}\label{sec:Math_framework}

\subsection{Model problem}

Let $\mathbb{U}$ and $\mathbb{V}$ be two Hilbert spaces endowed with inner products $(\cdot,\cdot)_{\mathbb{U}}$ and $(\cdot,\cdot)_{\mathbb{V}}$, respectively.
Let $b:\mathbb{U}\times\mathbb{V}\longrightarrow\mathbb{R}$ be a bilinear form, such that for any given continuous linear functional $\ell : \mathbb{V}\longrightarrow\mathbb{R}$, the following variational problem admits a unique solution:
\begin{equation}
    \left\{
    \begin{array}{ll}
    \text{Find } u^* \in \mathbb{U} \text{ such that} & \\
        b(u^*,v) = \ell(v),\;\forall\, v\in \mathbb{V}.
\end{array}
    \right.
    \label{eq:variational_form_problem}
\end{equation}
It is well-known (see, e.g.,~\cite[Theorem 25.9]{ErnGuermond2021_FiniteElementsII}) that problem~\eqref{eq:variational_form_problem} admits a unique solution whenever the following conditions hold:
\begin{alignat}{3}
&\textbf{Continuity:}& \quad &\exists\, M>0 \text{ such that } \sup_{v \in \mathbb{V} \setminus \{0\}} \frac{|b(u,v)|}{\|v\|_\mathbb{V}} \leq M \|u\|_\mathbb{U}, \quad &&\forall u \in \mathbb{U},\label{eq:continuity} \\
&\textbf{Injectivity:}& \quad &\exists\, \gamma_1>0 \text{ such that } \sup_{v \in \mathbb{V} \setminus \{0\}} \frac{|b(u,v)|}{\|v\|_\mathbb{V}} \geq \gamma_1 \|u\|_\mathbb{U}, \quad &&\forall u \in \mathbb{U}, \label{eq:inf-sup} \\
&\textbf{Surjectivity:}& \quad &\exists\, \gamma_2>0 \text{ such that } \sup_{u \in \mathbb{U} \setminus \{0\}} \frac{|b(u,v)|}{\|u\|_\mathbb{U}} \geq \gamma_2 \|v\|_\mathbb{V}, \quad &&\forall v \in \mathbb{V}. \label{eq:surjectivity}
\end{alignat}
As a result, the lower-bound constants coincide (i.e., $\gamma_1 
= \gamma_2 =:\gamma$) and the following robustness relation holds:
\begin{equation}\label{eq:dual_inf-sup}
\gamma \|u\|_\mathbb{U} \leq \sup_{v\in\mathbb{V}\setminus\{0\}} \frac{|b(u,v)|}{\|v\|_\mathbb{V}} \leq M \|u\|_\mathbb{U},\qquad u\in\mathbb{U}.
\end{equation}
\subsection{The Uzawa method}
Given $u\in\mathbb U$, we define the residual $r=r(u)\in\mathbb V$ as the unique solution to the following variational problem:
\begin{equation}\label{eq:residual}
(r,v)_{\mathbb{V}} = \ell(v) - b(u,v)=b(u^*-u,v), \quad \forall v\in \mathbb{V}.
\end{equation} Note that $r$ is the Riesz representative in $\mathbb{V}$ of the continuous linear functional $v\mapsto b(u^*-u,v)$.

Now, given the above $r=r(u)\in\mathbb{V}$, let us define $\delta=\delta(r)$ as solution to the following variational problem:
\begin{equation}\label{eq:delta}
(\delta,w)_{\mathbb{U}} = b(w,r), \quad \forall w\in \mathbb{U}.
\end{equation}
Analogously, $\delta$ is the Riesz representative in $\mathbb{U}$ of the continuous linear functional $w \mapsto b(w,r)$. Therefore, $\delta$ acts as the optimal correction direction in the trial space driven by the residual $r$.
\begin{proposition}\label{prop:equivalences}
Let $\left\{u^k\right\}\subset\mathbb U$ be a sequence and define $\left\{r^k\right\}:=\left\{r(u^k)\right\}\subset \mathbb V$ and 
$\left\{\delta^k\right\}=\left\{\delta(r^k)\right\}\subset\mathbb U$.
Then, the following statements are equivalent:
\begin{enumerate}
\item[(i)] $\left\{u^k\right\}\to u^*$ in $\mathbb U$.
\item[(ii)] $\left\{r^k\right\}\to 0_\mathbb V$ in $\mathbb V$.
\item[(iii)] $\left\{\delta^k\right\}\to 0_\mathbb U$ in $\mathbb U$.
\end{enumerate}
\end{proposition}
\begin{proof}
    If $\left\{u^k\right\}\to u^*$ in $\mathbb U$, then 
    \begin{align}
    \|r^k\|_\mathbb V=\sup_{v\in\mathbb V}{(r^k,v)_\mathbb V\over\|v\|_\mathbb V}
    =\sup_{v\in\mathbb V}{b(u^*-u^k,v)\over\|v\|_\mathbb V}\leq M \|u^*-u^k\|_\mathbb U,\tag{by~\eqref{eq:residual}~and~\eqref{eq:continuity} }
    \end{align}
    which proves (i) $\Rightarrow$ (ii).
    If $\{r^k\}\to 0_\mathbb V$ in $\mathbb V$, then
    \begin{align}
    \|\delta^k\|_\mathbb U=\sup_{w\in\mathbb U}{(\delta^k,w)_\mathbb U\over\|w\|_\mathbb U}
    =\sup_{w \in\mathbb U}{b(w,r^k)\over\|w\|_\mathbb U}\leq M \|r^k\|_\mathbb V,
    \tag{by~\eqref{eq:delta}~and~\eqref{eq:continuity} }
    \end{align}
    which proves (ii) $\Rightarrow$ (iii).
If $\{\delta^k\}\to 0_\mathbb U$ in $\mathbb U$, then
    \begin{align}
    \|u^*-u\|_\mathbb U\leq {1\over\gamma}\sup_{v\in\mathbb V}{b(u^*-u^k,v)\over\|v\|_\mathbb V}
    = {1\over\gamma}\|r^k\|_\mathbb V\leq {1\over\gamma^2}\sup_{w\in\mathbb U}{b(w,r^k)\over\|w\|_\mathbb U}={1\over\gamma^2}\|\delta^k\|_\mathbb U, 
    \tag{by~\eqref{eq:inf-sup}, \eqref{eq:residual}, \eqref{eq:dual_inf-sup}~and~\eqref{eq:delta}}
    \end{align}
    which proves (iii) $\Rightarrow$ (i).
\end{proof}

Given an initial guess $u^0\in\mathbb U$, the Uzawa method (see e.g., \cite{uzawa1958iterative, HamdBBEV2025,BennyChackoBEV2025}) constructs a sequence $\left\{u^k\right\}\subset \mathbb U$ whose elements are updated in the direction of $\delta^k=\delta(r^k)=\delta(r(u^k))$ with a step size $\rho>0$ as follows:
\begin{equation}\label{eq:uzawa}
\left\{
\begin{array}{l}
u^0\in \mathbb U\,,\\
u^{k+1} = u^k+\rho\delta^k=u^0+\rho\displaystyle\sum_{j=0}^k\delta^j\,.
\end{array}
\right.
\end{equation}

\paragraph{Gradient-descent viewpoint.}
The iterative scheme of the Uzawa method,
$$u^{k+1} = u^k + \rho \delta^k,$$
can be naturally interpreted as a residual-oriented gradient-descent iteration. Let us define the quadratic functional $J(u) := \frac{1}{2}\|r(u)\|_{\mathbb{V}}^2$. We formally define its gradient $\nabla J(u)$ as the Riesz representative in $\mathbb{U}$ of its Fréchet derivative $D J(u)$ that lies in the topological dual of $\mathbb{U}$ and satisfies
$$ (\nabla J(u), w)_{\mathbb{U}} = \langle D J(u), w\rangle_{\mathbb{U}^* \times \mathbb{U}} = \left. \frac{d}{d\epsilon} J(u + \epsilon w) \right|_{\epsilon=0}, \qquad \forall w \in \mathbb{U}.$$
A direct calculation demonstrates that the above is equal to $-b(w, r(u))$ for all $w \in \mathbb{U}$. Consequently, we obtain the identity $\nabla J(u) = -\delta(r(u))$. This relationship allows the Uzawa method, with step size $\rho > 0$, to be equivalently rewritten as
$$u^{k+1} = u^k - \rho \nabla J(u^k).$$

The following subsection discusses the step size tuning to guarantee convergence.

\subsection{Convergence of the Uzawa method}

Let $B:\mathbb{U}\longrightarrow \mathbb{V}$ and $B':\mathbb{V}\longrightarrow \mathbb{U}$ be the trial-to-test and test-to-trial operators inherited from the bilinear form $b(\cdot,\cdot)$ as follows: 
\begin{equation}
    (Bu,v)_\mathbb{V} = b(u,v) = (u,B'v)_\mathbb{U},\qquad u\in\mathbb{U},v\in\mathbb{V}.
\end{equation} 
Then,
$$
r^k=B(u^*-u^k) \qquad\hbox{ and }\qquad \delta^k=B'r^k=B'B(u^*-u^k).
$$
Defining the error $e^k := u^*-u^k$, the iterative scheme~\eqref{eq:uzawa} yields
\begin{align}\label{eq:contraction}
e^{k+1}=u^*-u^{k+1}= u^*-u^{k} -\rho\delta^k=(I-\rho B'B)(u^*-u^k) =(I-\rho B'B)e^k.
\end{align} 
\begin{proposition}\label{prop:contraction}
    The operator $I-\rho B'B$ is a contraction (hence, $\{u^k\}\to u^*$) whenever $\rho <2 \|B\|^{-2}$.
Indeed, the optimal choice for $\rho$ is $\rho^*:=2/(\|B\|^2+\|B^{-1}\|^{-2})$, with optimal contraction rate given by
    $$
\|I-\rho^*B'B\|={\|B\|^2-\|B^{-1}\|^{-2}\over\|B\|^2+\|B^{-1}\|^{-2}}.
$$
\end{proposition}
\begin{proof}
Because $B'B$ is self-adjoint, so is $I-\rho B'B$. Using~\cite[Proposition~6.9]{Brezis2010}, we have that $\|I-\rho B'B\|<1$ is satisfied provided 
$$
-\|u\|_\mathbb U^2 < \|u\|_\mathbb U^2 -\rho \|Bu\|_\mathbb V^2 <\|u\|_\mathbb U^2, \quad\forall u\in\mathbb U\setminus\{0_\mathbb U\}.
$$
The upper bound is always satisfied because $\rho>0$, while the lower bound holds when
$$
\rho<2\left({\|u\|_\mathbb U\over\|Bu\|_\mathbb V}\right)^2,\quad\forall u\in\mathbb U\setminus\{0_\mathbb U\}.
$$
The upper bound is minimized when $\|Bu\|_\mathbb V/\|u\|_\mathbb U$ attains its maximum value, yielding $\rho \leq 2\|B\|^{-2}$.
Moreover,
$$
\begin{array}{rl}
\|I-\rho B'B\|= & \displaystyle\max_{u\neq 0}\left| 1-\rho {\|Bu\|_{\mathbb V}^2\over \|u\|_\mathbb U^2}\right|=
\max\Big\{
\rho \|B\|^2-1\,~,\,1-\rho/\|B^{-1}\|^{2}\Big\}\\\\
= & 
\left\{
\begin{array}{ll}
\rho\|B\|^2-1 & \hbox{ if }\, \rho \ge 2\Big(\|B^{-1}\|^{-2}+\|B\|^2\Big)^{-1},\\\\
1-\rho/\|B^{-1}\|^{2} & \hbox{ if }\, \rho \le 2\Big(\|B^{-1}\|^{-2}+\|B\|^2\Big)^{-1},
\end{array}\right.
\end{array}
$$
whose minimum (as a function of $\rho$) is attained exactly at $\rho^*=2\Big(\|B^{-1}\|^{-2}+\|B\|^2\Big)^{-1}$. 
\end{proof}

\begin{remark}[Ideal convergence with graph inner product]\label{section:particular_optimal}
By equipping $\mathbb{U}$ with the graph inner product $(u,w)_\mathbb{U}:=(Bu,Bw)_\mathbb{V}$, we obtain $B'=B^{-1}$.
As a result,
\begin{equation}\label{Uzawa_iterative}
e^{k+1} = (1-\rho) e^{k},
\end{equation} 
which yields the optimal choice $\rho^*=1$ that reduces the iterative scheme to a single step.
\end{remark}

\subsection{Combining Ritz and Uzawa: three approaches}

Following the definitions of~\eqref{eq:residual} and~\eqref{eq:delta}, for an initial guess $u^{0} \in \mathbb{U}$, the Uzawa iterative method reads as:
\begin{equation}
    \left\{
    \begin{aligned}
        (r^{k},v)_{\mathbb{V}} &= \ell(v)  - b(u^k,v), &&\forall v\in\mathbb{V},\\
        (\delta^{k},w)_{\mathbb{U}} &= b(w,r^{k}), &&\forall w\in\mathbb{U},\\
        u^{k+1} &= u^k + \rho \delta^k.
    \end{aligned}
    \right.
\label{eq:Uzawa_continuous0}
\end{equation}
Now, to address the variational characterizations of $r^k$ and $\delta^k$, we reformulate them as optimization problems in the following three ways:
\subsubsection{Approach 1}\label{subsection:main_approach1}

We formulate the computation of the residual $r^k$ and the correction $\delta^k$ in terms of suitable Ritz minimizations within the functional spaces $\mathbb{V}$ and $\mathbb{U}$, respectively. We then construct the updated solution $u^{k+1}$ following the summation scheme indicated in \eqref{eq:uzawa}, which yields the following iterative method: given $u^0\in\mathbb{U}$,
\begin{equation}
    \left\{
    \begin{aligned}
        r^{k} &= \arg\min_{r\in \mathbb{V}} \frac{1}{2} \|r\|^2_{\mathbb{V}} - \ell(r) + b(u^k,r), \\[1ex]
        \delta^{k} &= \arg\min_{\delta\in \mathbb{U}} \frac{1}{2} \|\delta\|^2_{\mathbb{U}} - b(\delta,r^{k}), \\[1ex]
        u^{k+1} &= u^0 + \rho \delta^0 + \rho \delta^1 + \ldots + \rho\delta^k.
    \end{aligned}
    \right.
    \label{eq:Ritz-Uzawa_summation}
\end{equation}
Note that \eqref{eq:Ritz-Uzawa_summation} can be viewed as a generalization of the work on multilevel neural networks developed in \cite{aldirany2024multi} to (variational) problems using the Ritz method.
We emphasize that this approach does not require the bilinear form $b$ to be symmetric or positive definite.

\subsubsection{Approach 2} \label{subsection:main_approach2}

The function spaces defined by neural network architectures are generally not closed under addition. This implies that explicitly computing and storing the sum that defines $u^{k+1}$ in Approach 1 could become prohibitively expensive in memory as the iteration counter $k$ increases. To avoid this limitation, we reformulate the update step as an additional optimization problem within the ideal setting. Specifically, we replace the summation with the following $\mathbb{U}$-norm minimization problem:
\begin{equation}
    u^{k+1} = \displaystyle\arg\min_{u\in\mathbb{U}} \Vert u - (u^k + \rho \delta^k)\Vert_{\mathbb{U}}^2\,.
\end{equation}
Consequently, the new iterative method reads as follows: given $u^k\in\mathbb{U}$,
\begin{equation}\label{eq:approach2_exact}
    \left\{
    \begin{aligned}
        r^{k}    &= \arg\min_{r \in \mathbb{V}} \frac{1}{2} \|r\|^2_{\mathbb{V}} - \ell(r) + b(u^k,r), \\[1ex]
        \delta^{k} &= \arg\min_{\delta \in \mathbb{U}} \frac{1}{2} \|\delta\|^2_{\mathbb{U}} - b(\delta,r^{k}), \\[1ex]
        u^{k+1}  &= \arg\min_{u \in \mathbb{U}} \| u - (u^k + \rho \delta^k) \|_{\mathbb{U}}^2.
    \end{aligned}
    \right.
\end{equation} 
As with Approach 1, this ideal formulation has the advantage that the integrands of the loss functions associated with $r^k$ and $\delta^k$ tend to zero at convergence. Furthermore, defining $u^{k+1}$ via minimization avoids the need to dynamically expand the network architecture during the iterative process.

\subsubsection{Approach 3} \label{subsection:main_approach3}

Replacing $\delta^k$ with $\frac{u^{k+1}-u^k}{\rho}$ in \eqref{eq:Uzawa_continuous0}, simplifying $u^{k+1}$ and rewriting the resulting equation as a minimization problem, we obtain the following two-step minimization scheme: given $u^k\in\mathbb U$,

\begin{equation}
    \left\{
    \begin{aligned}
        r^{k}   &= \arg\min_{r \in \mathbb{V}} \frac{1}{2} \|r\|^2_{\mathbb{V}} - \ell(r) + b(u^k,r), \\[2ex]
        u^{k+1} &= \arg\min_{u \in \mathbb{U}} \frac{1}{2} \|u\|^2_{\mathbb{U}} - (u^{k},u)_{\mathbb{U}} - \rho b(u,r^{k}).
    \end{aligned}
    \right.
    \label{eq:Ritz-Uzawa_two_eqs}
\end{equation} The integrand of the second minimization does not tend to zero at converge.
In Section~\ref{section:Simplifications}, we show how this approach relates to the Ritz, double Ritz, and adjoint Ritz methods \cite{EYu2018DeepRitz,uriarte2023DoubleRitz}.
\subsection{Simplifications under some usual formulations}\label{section:Simplifications}

\subsubsection{Strong formulation}\label{section:PINNs}
For $\mathbb{V}=L^2$ equipped with the standard inner product, we consider the following framework:
\begin{itemize}
    \item $B:\mathbb{U}\longrightarrow L^2$ is the available PDE operator and $f\in L^2$ is the source.
    \item $b(u,v)=(Bu,v)_{L^2}$ is the bilinear form and $\ell(v)=(f,v)_{L^2}$ is the right-hand side for all $u\in\mathbb{U}$ and all $v\in L^2$. Recalling Remark~\ref{section:particular_optimal}, by equipping $\mathbb{U}$ with the graph inner product $(u,w)_\mathbb{U} := (Bu,Bw)_{L^2}$, we have that $\rho^*=1$ is an optimal choice in Uzawa.
    \item The Ritz minimization for $r^k$ is unnecessary, as $r^k$ is explicitly available: $r^k=f-Bu^k$.
    \item The resulting variational problem on $\delta^k$, 
    $$\underbrace{(\delta^k, w)_\mathbb{U}}_{(B\delta^k,Bw)_{L^2}} = \underbrace{b(w,r^k)}_{(Bw,f-Bu^k)_{L^2}},\quad \forall w\in\mathbb{U},$$ can be reformulated as the following $L^2$-norm minimization:
    $$ \delta^k = \arg\min_{\delta\in\mathbb{U}} \Vert B\delta + Bu^k - f\Vert_{L^2}^2.$$
\end{itemize}
As a result, Approaches 1-3 simplify as follows:
    \begin{enumerate}
        \item[] \textbf{Approach 1 (S1).} Given $u^0\in\mathbb{U}$,\begin{equation}
    \left\{
    \begin{aligned}
        \delta^{k} &= \arg\min_{\delta \in \mathbb{U}} \| B\delta + Bu^k - f \|_{L^2}^2, \\[2ex]
        u^{k+1}    &= u^0 + \delta^0 + \delta^1 + \ldots + \delta^k.
    \end{aligned}
    \right.
    \label{eq:UzawaPINNs_ap1}
\end{equation}
This iterative scheme might be viewed as the multilevel neural network method proposed in \cite{aldirany2024multi}.
\item[] \textbf{Approach 2 (S2).} Given $u^k\in\mathbb U$,\begin{equation}
    \left\{
    \begin{aligned}
        \delta^{k} &=  \displaystyle\arg\min_{\delta\in \mathbb{U}}\Vert B\delta + Bu^k - f\Vert_{L^2}^2,\\[2ex]
        u^{k+1} &= \displaystyle\arg\min_{u\in\mathbb{U}} \Vert B\delta^k + Bu^k - Bu\Vert_{L^2}^2.
\end{aligned}
    \right.
\end{equation}
         \item [] \textbf{Approach 3 (S3).} We obtain a simplified version of the the deep double Ritz method proposed in \cite{uriarte2023DoubleRitz}:
\begin{align}
         u^{k+1} &=  \displaystyle\arg\min_{u\in \mathbb{U}} \frac{1}{2} \|Bu\|^2_{L^2} - (f,Bu)_{L^2}.
\label{eq:Uzawa_continuous3}
\end{align}
\end{enumerate}

\subsubsection{Weak formulation}

In symmetric and positive-definite problems, the framework reads within our Ritz-Uzawa scheme as follows:
\begin{itemize}
    \item $\mathbb{U}=\mathbb{V}$ are equipped with the inner product induced by the bilinear form $b(\cdot,\cdot)$. Thus, $B$ is the identity operator, and $\mathbb{U}$ can be read as equipped with the graph inner product (recall Remark~\ref{section:particular_optimal}).
Then, $\rho^*=1$.
    \item We have $\mathbb{V}\ni r^k = \delta^k = u^* - u^k\in\mathbb{U}$.
Then, we can remove the formulations related to $\delta^k$ and directly consider $u^{k+1} = u^k + r^k$.
\end{itemize} As a result, Approaches 1-3 simplify as follows:
    \begin{enumerate}
        \item[] \textbf{Approach 1 (W1).} Given $u^0\in\mathbb U$,\begin{equation}
    \left\{
    \begin{aligned}
        r^{k} &=  \displaystyle\arg\min_{r\in \mathbb{V}}\frac{1}{2} b(r,r) + b(u^k,r)-\ell(r),\\[2ex]
        u^{k+1} &= u^0 + r^0 + r^1 + \cdots + r^k.
\end{aligned}
    \right.
    \label{eq:Uzawa_continuous1}
\end{equation}
        \item[] \textbf{Approach 2 (W2).} Given $u^k\in\mathbb U$,\begin{equation}
    \left\{
    \begin{aligned}
        r^{k} &=  \displaystyle\arg\min_{r\in \mathbb{V}}\frac{1}{2} b(r,r) + b(u^k,r)-\ell(r),\\[2ex]
        u^{k+1} &= \displaystyle\arg\min_{u\in\mathbb{U}} b(r^k + u^k - u, r^k + u^k - u).
\end{aligned}
    \right.
    \label{eq:Uzawa_continuous2}
\end{equation}
         \item [] \textbf{Approach 3 (W3).} We recover the traditional Ritz minimization,
\begin{equation}
        u^{k+1} =  \displaystyle\arg\min_{u\in \mathbb{U}}\frac{1}{2} b\left(u,u\right)-\ell(u).
\end{equation}
    \end{enumerate}

\subsubsection{Ultra-weak formulation}\label{section:ultraweak}

In ultra-weak formulations, $\mathbb{U}=L^2$, $b(u,v)=(u,B'v)_{L^2}$ and $\ell(v)=(f,v)_{L^2}$ for all $u\in L^2$ and all $v\in \mathbb{V}$. Under this framework, the following properties hold:
\begin{itemize}
    \item By equipping $\mathbb{V}$ with the adjoint-graph inner product $(\cdot,\cdot)_{\mathbb{V}} := (B'\cdot,B'\cdot)_{L^2}$, we obtain $\rho^*=1$ as the optimal step size (recall Remark~\ref{section:particular_optimal}).
    \item The Ritz minimization for $\delta^k$ is explicitly available as $\delta^k=B'r^k$, thus avoiding the need to compute it numerically.
\end{itemize} 
As a result, Approaches 1-3 simplify as follows:
    \begin{enumerate}
        \item[] \textbf{Approach 1 (U1).} Given $u^0\in\mathbb U$,\begin{equation}
    \left\{
    \begin{aligned}
        r^{k} &=  \displaystyle\arg\min_{r\in \mathbb{V}}\frac{1}{2} (B'r,B'r)_{L^2} + b(u^k,r)-\ell(r),\\[2ex]
        u^{k+1} &= u^0 + B'r^0 + B'r^1 + \cdots + B'r^k.
\end{aligned}
    \right.
    \label{eq:UzawaUltraweak_ap1}
\end{equation}
        \item[] \textbf{Approach 2 (U2).} Given $u^k\in\mathbb U$,\begin{equation}
    \left\{
    \begin{aligned}
        r^{k} &=  \displaystyle\arg\min_{r\in \mathbb{V}}\frac{1}{2}(B'r,B'r)_{L^2} + b(u^k,r)-\ell(r),\\[2ex]
        u^{k+1} &= \displaystyle\arg\min_{u\in\mathbb{U}} \Vert B'r^k+ u^k - u\Vert_{L^2}^2.
\end{aligned}
    \right.
    \label{eq:Uzawa_continuous5}
\end{equation}
         \item [] \textbf{Approach 3 (U3).} We recover the Adjoint Ritz minimization with post-processing mentioned in \cite{uriarte2023DoubleRitz}.
Given $u^{k}\in\mathbb U$, \begin{equation}
        \left\{
    \begin{aligned}
        r &=  \arg\min_{r\in \mathbb{V}} \frac{1}{2} \Vert B'r\Vert_{L^2}^2 - \ell(r),\\[2ex]
        u^{k+1} &= B'r.
\end{aligned}
    \right.
\end{equation}
\end{enumerate}
\subsection{The inexact Ritz-Uzawa method} \label{subsection:practical_setup}

In practice, mainly due to the imprecision of minimization algorithms, the ideal minimizations described in Approaches 1, 2 and 3 cannot be solved exactly. Optimization algorithms introduce approximation errors, meaning we only recover inexact versions of the minimizing sequences. Therefore, we introduce $r^k_\varepsilon$, $\delta^k_\varepsilon$, and $u^k_\varepsilon$ to represent the quasi-minimizers at each step, where $\varepsilon>0$ is a tolerance parameter controlling the relative errors of the optimization \cite{brevis2022neural}.

\paragraph{Inexact Approach 1.}
For the summation scheme, the inexact method is defined as follows: given $u^0\in\mathbb U$,
\begin{equation}\label{eq:approach1_inexact}
    \left\{
    \begin{aligned}
        r^k_\varepsilon &\approx r^{k} =  \displaystyle \arg\min_{r\in \mathbb{V}}\frac{1}{2}\left\|r\right\|^2_{\mathbb{V}} - \ell(r)  + b(u^k,r),
\\[2ex]
        \delta^k_{\varepsilon} &\approx\delta^{k} =  \displaystyle\arg\min_{\delta\in \mathbb{U}}\frac{1}{2}\left\|\delta\right\|^2_{\mathbb{U}} - b(\delta,r^{k}_\varepsilon),\\[2ex]
        u^{k+1} &= u^0 + \rho \delta^0_{\varepsilon} + \rho \delta^1_{\varepsilon} + \cdots + \rho \delta^k_{\varepsilon},
    \end{aligned}
    \right.
\end{equation} 
where we request the optimization process to be sufficiently accurate to bound the relative errors of the residual and the correction as follows:
\begin{equation}\label{eq:relative1}
    {\|r^k_\varepsilon-r^k\|_\mathbb{V}\over\|r^k\|_\mathbb{V}}\leq \varepsilon \qquad\hbox{ and }\qquad
    {\|\delta^k_{\varepsilon}-\delta^k\|_\mathbb{U}\over\|\delta^k\|_\mathbb{U}}\leq \varepsilon.
\end{equation}
The convergence of this inexact method is established in the following theorem, based on the result from~\cite[Theorem~3.1]{BennyChackoBEV2025}.

\begin{theorem}\label{thm:approach1_inexact}
    If $\rho<2\|B\|^{-2}$ and~\eqref{eq:relative1} hold for $\varepsilon>0$ satisfying $\varepsilon^2+2\varepsilon< (1-\|I-\rho B'B\|)/\rho \|B\|^2$, then the scheme defined in~\eqref{eq:approach1_inexact} is such that $\{r^k_\varepsilon\}\to 0_\mathbb{V}$, 
    $\{\delta^k_\varepsilon\}\to 0_\mathbb{U}$ and $\{u^k\}\to u^*$.
\end{theorem}
\begin{proof}
     See Appendix~\ref{sec:proof_thm1}. 
\end{proof} 

\paragraph{Inexact Approach 2.}
When replacing the summation of  the inexact Approach 1 with a minimization of the $\mathbb{U}$-norm to avoid dynamically adding previous iteration outcomes, the inexact method reads as follows: given $u^k\in\mathbb U$,
\begin{equation}\label{eq:approach2_inexact}
    \left\{
    \begin{aligned}
        r^k_\varepsilon&\approx r^{k} =  \displaystyle \arg\min_{r\in \mathbb{V}}\frac{1}{2}\left\|r\right\|^2_{\mathbb{V}} - \ell(r)  + b(u^k_\varepsilon,r), \\[2ex]
        \delta^k_{\varepsilon}&\approx\delta^{k} =  \displaystyle\arg\min_{\delta\in \mathbb{U}}\frac{1}{2}\left\|\delta\right\|^2_{\mathbb{U}} - b(\delta,r^{k}_\varepsilon),\\[2ex]
        u^{k+1}_\varepsilon&\approx    u^{k+1} = \displaystyle\arg\min_{u\in\mathbb{U}} \Vert u - (u^k_\varepsilon + \rho \delta^k_\varepsilon)\Vert_{\mathbb{U}}^2\,.
    \end{aligned}
    \right.
\end{equation} 
The integrand of this final minimization tends to zero if we use a warm-start strategy, initializing the network parameters of $u^{k+1}_\varepsilon$ with those from $u^k_\varepsilon$. In addition to the bounds in~\eqref{eq:relative1}, we must also control the following relative error of this final update:
\begin{equation}\label{eq:u_relative_err}
{\|u^{k+1}_\varepsilon-u^{k+1}\|_\mathbb{U}\over \|\rho\delta_\varepsilon^k\|_\mathbb{U}}\leq \varepsilon.
\end{equation}
If $u^k_\varepsilon$ is poorly trained, the updated function $u^{k+1}_\varepsilon$ may not improve the solution, and the error could even increase. The following theorem establishes the convergence of this approach under specific tolerance conditions.

\begin{theorem}\label{thm:approach2_inexact}
    If $\rho<2\|B\|^{-2}$ and~\eqref{eq:relative1}-\eqref{eq:u_relative_err} hold for $\varepsilon>0$ satisfying $\varepsilon^3+3\varepsilon^2+3\varepsilon< (1-\|I-\rho B'B\|)/\rho \|B\|^2$, then the inexact method~\eqref{eq:approach2_inexact} satisfies $\{r^k_\varepsilon\}\to 0_\mathbb{V}$, $\{\delta^k_\varepsilon\}\to 0_\mathbb{U}$ and $\{u^k_\varepsilon\}\to u^*$.
\end{theorem}
\begin{proof}
     See Appendix~\ref{sec:proof_thm2}. 
\end{proof}

\paragraph{Inexact Approach 3.}
The inexact version of this two-step scheme corresponds to the method analyzed in~\cite{BennyChackoBEV2025}. Introducing the neural network approximations $r^k_\varepsilon$ and $u^k_\varepsilon$, the method reads:
\begin{equation}\label{eq:approach3_inexact}
    \left\{
    \begin{aligned}
        r^{k}_\varepsilon &\approx  \displaystyle \arg\min_{r\in \mathbb{V}}\frac{1}{2}\left\|r\right\|^2_{\mathbb{V}} - \ell(r)  + b(u^k_\varepsilon,r), \\[2ex]
        u^{k+1}_\varepsilon &\approx  \displaystyle\arg\min_{u\in \mathbb{U}}\frac{1}{2}\left\|u\right\|^2_{\mathbb{U}} - (u^{k}_\varepsilon,u)_{\mathbb{U}} - \rho b(u,r^{k}_\varepsilon).
\end{aligned}
    \right.
\end{equation}
As established in~\cite[Theorems 3.1 and 3.2]{BennyChackoBEV2025}, this iterative scheme converges provided the inexact inner updates satisfy relative error bounds analogous to~\eqref{eq:relative1} and~\eqref{eq:u_relative_err}, ensuring they consistently move in the correct descent direction.

\section{Neural network setting}\label{sec:architecture_and_training}

\subsection{Architecture} To construct approximations of the trial and test spaces $\mathbb{U}$ and $\mathbb{V}$, we parameterize the solutions using neural networks. In order to efficiently capture a broad range of frequencies, we employ a modified architecture defined as follows:

\begin{itemize}
    \item[\textbf{A)}] \textbf{First layer (Fourier feature mapping).} Let $\mathbf{x} \in \Omega \subset \mathbb{R}^d$ be the input vector.  To capture high-frequency features \cite{tancik2020fourier}, we define the mapping $\gamma_{\boldsymbol{\kappa}}: \Omega \to \mathbb{R}^n$ using trainable weights $\mathbf{w}_j \in \mathbb{R}^d$, biases $b_j \in \mathbb{R}$ for $j=1,2,...,n$, and a frequency vector $\boldsymbol{\kappa}=(\kappa_1,...,\kappa_n) \in \mathbb{R}^n$ that will be specified in Section~\ref{subsec:NCPSD}). Then, the output of the first layer is defined as
    \begin{equation}
        \mathbf{z}_1 := \gamma_{\boldsymbol{\kappa}}(\mathbf{x}) = \big( 
        \sin(\kappa_1 (\mathbf{w}_1^\top \mathbf{x} + b_1)), \dots,
        \sin(\kappa_n (\mathbf{w}_n^\top \mathbf{x} + b_n)) \big).
    \end{equation}
    
    \item[\textbf{B)}] \textbf{Hidden layers.} Let $\sigma$ be a non-linear activation function (e.g., $\tanh$). For each layer $l$, with trainable weights $\mathbf{W}_l \in \mathbb{R}^{n \times n}$ and biases $\mathbf{b}_l \in \mathbb{R}^n$, we define the remaining intermediate hidden outputs as
    \begin{equation}
        \mathbf{z}_l := \sigma(\mathbf{W}_l \mathbf{z}_{l-1} + \mathbf{b}_l).
    \end{equation}
\end{itemize}

\paragraph{Output generator functions.} If necessary, to strongly satisfy the homogeneous Dirichlet boundary conditions, we multiply the network output $\mathbf{z}_L$ element-wise by a smooth cut-off function $\xi(\mathbf{x})$ that vanishes on the portion of the boundary where Dirichlet boundary conditions are imposed \cite{BerronePintore2023Enforcing}. This defines the functions $\{\phi_j\}_{j=1}^n$ as
\begin{equation}
    \phi_j(\mathbf{x}) = (\mathbf{z}_L(\mathbf{x}))_j \cdot \xi(\mathbf{x}).
\end{equation}

\paragraph{Linear combination.}
The final output is the linear combination between the above functions $\{\phi_j\}_{j=1}^n$ and a set of output weights denoted $\mathbf{w}_{out} := (w_{out, j})_{j=1}^n \in \mathbb{R}^n$:
\begin{equation}
    v_\theta(\mathbf{x}) = \sum_{j=1}^n w_{out, j} \, \phi_j(\mathbf{x}).
\end{equation}
Thus, the set of trainable parameters is $\theta:=\{ \mathbf{w}_{out} \} \cup \theta_H$, where $\theta_H := \{ \{\mathbf{w}_j, b_j \}_{j=1}^n , \bigcup_{l=2}^L \{ \mathbf{W}_l, \mathbf{b}_l \} \}$ denotes the subset of hidden parameters.

\subsection{Parameter initialization via spectral analysis} \label{subsec:NCPSD}

Neural networks suffer from \textit{spectral bias} \cite{Cao2021Spectral, Rahaman2019Spectral}, effectively acting as low-pass filters that favor smooth approximations over high-frequency components. While Fourier Feature Mapping mitigates this issue through tunable frequency parameters $\boldsymbol{\kappa}=(\kappa_1, \dots, \kappa_n)$, standard initialization strategies often fail to target the specific spectral content of the solution. To address this, we employ a data-driven tuning strategy based on the \textit{power spectral density} (PSD) (see, e.g., \cite{brunton2022datadriven, kim2024cumulative}).

{Given a function $g$ evaluated on a discrete computational grid (which in our framework represents either a residual or a Riesz representative)}, we compute its discrete Fourier transform (DFT) using the Fast Fourier Transform (FFT). To quantify the spectral energy distribution up to a frequency threshold $\eta>0$, we define the \textit{normalized cumulative power spectral density} (NCPSD) as follows:
\begin{equation} \label{eq: discrete_NCPSD}
    \text{NCPSD}_s(g)(\eta) := \frac{\displaystyle \sum_{|\omega_m| \le \eta} (1 + |\omega_m|^2)^s |\hat{g}_m|^2}{\displaystyle \sum_{m=1}^{M} (1 + |\omega_m|^2)^s |\hat{g}_m|^2},
\end{equation}
where $\hat{g}_m$ denotes the $m$-th coefficient of the DFT of $g$, $\omega_m$ is the corresponding discrete angular frequency, {$M$ is the number of samples used to discretize $g$}, and $s \in \mathbb{Z}$ is an integer index that dictates frequency weighting.

The choice of the weighting factor $(1 + |\omega_m|^2)^s$ is motivated by Plancherel's theorem. Indeed, the Sobolev norm $\|g\|_{H^s(\mathbb{R}^d)}^2$ of a generic function is equivalent to a frequency-weighted $L^2$-norm:
\begin{equation}\label{eq: continuous_plancherel}
    \|g\|_{H^s(\mathbb{R}^d)}^2 \cong \int_{\mathbb{R}^d }(1+|\omega|^2)^s |\hat{g}(\omega)|^2 d\omega.
\end{equation}
Thus, the integer index $s$ dictates the regularity regime: non-negative values ($s \geq 0$) correspond to standard Sobolev spaces of functions, whereas negative values ($s < 0$) characterize dual spaces of distributions. Although the spectral representation of $H^s$ norms on bounded domains with imposed boundary conditions becomes involved, the discrete estimator defined in \eqref{eq: discrete_NCPSD} acts as a computationally tractable surrogate. {In the context of a Poisson-type elliptic problem, we select $s$ explicitly depending on the variational formulation being solved:
\begin{itemize}
    \item \textbf{Strong formulation.} The residual is considered as an element of $L^2$ ($s=0$), while the Riesz representative of the correction in the trial space is evaluated in $H^2$ ($s=2$).
    \item \textbf{Weak formulation.} The residual is considered as an element of the dual test space $H^{-1}$ ($s=-1$), while the Riesz representative of the correction in the trial space is evaluated in $H^1_0$ ($s=1$).
    \item \textbf{Ultra-weak formulation.} The residual is considered as an element of the dual test space $H^{-2}$ ($s=-2$), while the Riesz representative of the correction in the trial space is evaluated in $L^2$ ($s=0$).
\end{itemize}
}
By analyzing the NCPSD, we identify the effective bandwidth $[\omega_{\min}, \omega_{\max}]$ containing a target energy fraction $\alpha \in (0,1)$ (e.g., $\alpha = 0.05$). To uniquely define this interval and discard negligible energy at both extremes, we truncate the tails of the distribution symmetrically. Specifically, we choose the cut-off frequencies $\omega_{\min}$ and $\omega_{\max}$ such that:
\begin{equation}
    \text{NCPSD}_s(g)(\omega_{\min}) = \alpha \quad \text{and} \quad \text{NCPSD}_s(g)(\omega_{\max}) = 1-\alpha.
\end{equation} 
To ensure a multiscale representation, the frequency scalars $\kappa_j$ are sampled from a log-uniform distribution over this interval, i.e.,
\begin{equation}
    \ln(\kappa_j) \sim \mathcal{U}(\ln(\omega_{\min}), \ln(\omega_{\max})).
\end{equation}

On the other hand, regarding the spatial weights $\mathbf{w}_j$ and biases $b_j$, we follow the initialization scheme proposed for implicit neural representations \cite{sitzmann2020implicit}. Specifically, $\mathbf{w}_j$ are drawn from a uniform distribution $\mathcal{U}(-1,1)$ to ensure spectral coverage, while $b_j$ are sampled uniformly in $(-\pi, \pi)$. This combination guarantees that the output generators $\{\phi_j\}_{j=1}^n$ possess the necessary resolving power to approximate the dominant error modes from the onset of training. Finally, for the parameters of the subsequent hidden layers, we employ the Glorot uniform initialization when using the $\tanh$ activation function, and the He uniform initialization when employing $\text{ReLU}^3$.

\subsection{Training: hybrid LS/Adam optimization with normalization}

The training process aims to find an optimal set of parameters $\theta$ that minimize a given loss functional $\mathcal{L}(u_\theta)$.
To accelerate convergence, we employ a hybrid optimization strategy that decouples the linear and non-linear parameters (cf. \cite{uriarte2025vpinnsLS, Cyr2019RobustTA}).

\paragraph{Least squares step.}
For a fixed set of hidden trainable parameters $\theta_H$, the functions $\{\phi_j\}_{j=1}^n$ are fixed.
We notice that in every case (i.e., strong, weak, ultra-weak) the involved loss function is quadratic with respect to the linear output weights $\mathbf{w}_{out}$.
Hence, each loss can be expressed in algebraic form as:
\begin{equation}
    \mathcal{L}(\mathbf{w}_{out}) = \frac{1}{2} \mathbf{w}_{out}^\top \mathbf{H} \mathbf{w}_{out} - \mathbf{f}^\top \mathbf{w}_{out} + q,
\end{equation}
where $\mathbf{H} \in \mathbb{R}^{n \times n}$ is a Hessian matrix, $\mathbf{f} \in \mathbb{R}^n$ is a load vector, and $q \in \mathbb{R}$ is a scalar. When $\mathbf{H}$ is non-singular, minimizing this quadratic form is equivalent to solving the linear system $\mathbf{H} \mathbf{w}_{out} = \mathbf{f}$.

\paragraph{Output generators normalization.}
The Hessian matrix $\mathbf{H}$ may be singular due to linear dependencies of the output generators.
Standard Tikhonov regularization $(\mathbf{H} + \lambda \mathbf{I})$ can be ineffective if the spectrum of $\mathbf{H}$ is arbitrary, as the regularization parameter $\lambda$ might disproportionately affect relevant eigenvalues or fail to stabilize small ones.
To address this, we apply \textit{diagonal scaling} (normalization) to bound the maximum eigenvalue of the system.
Let $s_j = \sqrt{\mathbf{H}_{jj}}$ be the scaling factors and $\mathbf{S} = \text{diag}(s_1, \dots, s_n)$.

The normalized system is given by:
\begin{equation}
    \tilde{\mathbf{H}} \tilde{\mathbf{w}} = \tilde{\mathbf{f}}, \quad \text{with} \quad \tilde{\mathbf{H}} = \mathbf{S}^{-1}\mathbf{H}\mathbf{S}^{-1}, \quad \tilde{\mathbf{f}} = \mathbf{S}^{-1}\mathbf{f}.
\end{equation}
This normalization guarantees that the maximum eigenvalue is bounded by the network width ($\lambda_{\max}(\tilde{\mathbf{H}}) \leq \text{Tr}(\tilde{\mathbf{H}}) = n$). Consequently, this spectral bound ensures that the regularization term $\lambda \mathbf{I}$ applied to $\tilde{\mathbf{H}}$ preserves most of the information of the eigenmodes associated with eigenvalues larger than $\lambda$.
In this way, we solve $(\tilde{\mathbf{H}} + \lambda \mathbf{I})\tilde{\mathbf{w}} = \tilde{\mathbf{f}}$ and recover the initial weights via $\mathbf{w}_{out} = \mathbf{S}^{-1}\tilde{\mathbf{w}}$.

\paragraph{Adam step.}
After updating $\mathbf{w}_{out}$, we fix the linear layer and perform a gradient descent step (Adam) on the hidden parameters $\theta_H$ to minimize $\mathcal{L}(u_{\theta_H})$. Note that, with a slight abuse of notation, we use $\mathcal{L}$ throughout this section to denote the loss evaluated with respect to the specific subset of parameters (or functions) being optimized at each respective step. 

We describe this process across multiple epochs, Algorithm~\ref{alg:hybrid_training_loop} introduces the superscript $(k)$ to denote the state of the parameters at the $k$-th epoch. Specifically, the least-squares step at epoch $k$ utilizes the hidden parameters $\theta_H^{(k-1)}$ from the previous iteration to generate the basis functions and compute the updated linear weights $\mathbf{w}_{out}^{(k)}$. Subsequently, the Adam step uses these optimal weights to update the hidden representations to $\theta_H^{(k)}$. Finally, a concluding least-squares step is performed after the optimization loop finishes.

\begin{algorithm}[H]
\caption{Full hybrid LS/Adam training loop}
\begin{algorithmic}[1]
    \State \textbf{Input:} Initial hidden parameters $\theta_H^{(0)}$, regularization $\lambda$, total epochs $N_{epochs}$.
    \For{$k = 1$ \textbf{to} $N_{epochs}$}
        \State \textbf{1. Least-squares step (Output Layer):}
        \State \quad Generate basis $\{\phi_j\}$ using $\theta_H^{(k-1)}$ and assemble $\mathbf{H}, \mathbf{f}$ such that $\mathcal{L} \approx \frac{1}{2}\mathbf{w}^\top \mathbf{H} \mathbf{w} - \mathbf{f}^\top \mathbf{w}$.
        \State \quad Compute scaling $\mathbf{S} = \text{diag}(\sqrt{H_{jj}})$.
        \State \quad Solve $(\mathbf{S}^{-1}\mathbf{H}\mathbf{S}^{-1} + \lambda \mathbf{I})\tilde{\mathbf{w}} = \mathbf{S}^{-1}\mathbf{f}$.
        \State \quad Update output weights: $\mathbf{w}_{out}^{(k)} \gets \mathbf{S}^{-1}\tilde{\mathbf{w}}$.
        
        \State \textbf{2. Adam Step (Hidden Layers):}
        \State \quad Compute $\nabla_{\theta_H} \mathcal{L}$ using the updated $\mathbf{w}_{out}^{(k)}$.
        \State \quad Update hidden parameters: $\theta_H^{(k)} \gets \text{Adam}(\theta_H^{(k-1)}, \nabla_{\theta_H} \mathcal{L})$.
    \EndFor
    
    \State \textbf{3. Final least-squares step:}
    \State \quad Re-assemble $\mathbf{H}$ and $\mathbf{f}$ using the final hidden parameters $\theta_H^{(N_{epochs})}$.
    \State \quad Solve for the final optimal weights $\mathbf{w}_{out}^*$ following the scaled procedure (Lines 5-7).
    
    \State \textbf{Output:} Optimized parameters $\theta_H^* = \theta_H^{(N_{epochs})}$ and $\mathbf{w}_{out}^*$.
\end{algorithmic}
\label{alg:hybrid_training_loop}
\end{algorithm}

\section{Loss discretization} \label{sec:Loss_Discretization} 

All of the above formulations consider losses that require the evaluation of integrals over the computational domain $\Omega$. As shown in \cite{Rivera2022}, employing a deterministic quadrature rule for the deep Ritz method (DRM) can lead to catastrophic overfitting. The loss functions used in Approach~1 (see Section~\ref{subsection:main_approach1}), Approach~2 (see Section~\ref{subsection:main_approach2}) and Approach~3 (see Section~\ref{subsection:main_approach3}) are based on the DRM. Furthermore, it is established in \cite{taylor2025stochastic} that even in the best-case scenario where the optimizer converges to a local minimum, the bias inherent in quadrature rules can yield a poor approximation of the solution. To mitigate these issues, we employ unbiased stochastic quadrature rules. Now, we recall the discussion given in \cite{taylor2025stochastic}.

\subsection{Vanilla Monte Carlo}
This method approximates the integral by averaging the values of the integrand at a set of $N$ points $\{x_n\}_{n=1}^{N}$, sampled independently from a uniform distribution on $\Omega$. The approximation is given by:
\begin{align*}
    \int_{\Omega} L(x)dx \approx \frac{|\Omega|}{N}\sum_{n=1}^{N} L(x_n).
\end{align*}
This quadrature rule provides an unbiased estimator of the integral, i.e., $$\mathbb{E}\left [\frac{|\Omega|}{N}\sum_{n=1}^{N} L(x_n)\right] = \int_{\Omega} L(x)dx.$$ Furthermore, its variance is given by:
\begin{equation}
    \operatorname{Var}\left [ \frac{|\Omega|}{N}\sum_{n=1}^{N} L(x_n)\right] = \frac{|\Omega|^2}{N} \operatorname{Var}[L(x)] =\frac{|\Omega|}{N} \int_{\Omega} (L(x))^2 \, dx - \frac{1}{N} \left( \int_{\Omega} L(x) \, dx \right)^2.
\end{equation}
\subsection{Stratified Monte Carlo methods}
In this approach, the integration domain $\Omega$—here considered a 1D interval for simplicity—is partitioned into $K$ disjoint subintervals, such that $\Omega = I_1 \sqcup \dots \sqcup I_K$. This methodology can be extended to higher-dimensional meshes. The integral over $\Omega$ is decomposed into a sum, where each term is mapped to a common reference interval $I=(-1,1)$:
\begin{align*}
    \int_\Omega L(x) dx  = \sum_{i=1}^K \int_{I_i} L(x) dx  = \sum_{i=1}^K |I_i| \int_{I} L_i(\hat{x}) d\hat{x}.
\end{align*}
Here, $L_i(\hat{x})$ denotes the transformed integrand corresponding to the subinterval $I_i$. The samples within each element must be chosen independently and identically distributed (i.i.d.) and must remain independent across the different subintervals. To approximate the integral over the reference interval, we consider the following unbiased stochastic quadrature rules applied to a generic function $L(\hat{x})$.
\subsubsection*{Order-3 unbiased quadrature rule ($P_3$)}
A random point $x_1$ is sampled from the interval $(0,1)$ according to the probability density function $\mathcal{P}(x) = 3x^2$. The integral is then approximated by the following rule:
\begin{align*}
    \int_I L(\hat{x}) d\hat{x} \approx \frac{L(x_1) - 2L(0) + L(-x_1)}{3x_1^2} + 2L(0), \quad \text{where } x_1 \sim \mathcal{P}(x)=3x^2 \text{ on } (0,1).
\end{align*}
Consequently, for a partition of $K$ subintervals, the total number of collocation points is $N_K = 3K$. The quadrature rule is exact for cubic functions (see~\cite{Siegel1985UnbiasedMontecarlo}) and, when $L$ is $\mathcal{C}^4$, the variance scales as $\mathcal{O}((N_K)^{-9})$ (cf.~\cite{taylor2025stochastic}).

\subsection{Variance analysis}
We analyze the asymptotic behavior of the variance of the stochastic gradient estimators for the three proposed formulations within Approach~1. A key distinction arises between the strong formulation and the variational (weak and ultra-weak) formulations regarding the {\em passive variance reduction} properties of the gradient of the loss: while in the strong formulation we observe that the variance of the gradient of the loss tends to zero as we converge, this is not the case for the weak and ultra-weak formulations.

\paragraph{Strong formulation (S1).}
In the strong formulation (see Section \ref{section:PINNs}), the residual of the current iterate $u^k$ is explicitly given by $r^k = f - Bu^k$. The correction $\delta^k$ is obtained by minimizing the $L^2$-norm of the discrepancy between $B\delta$ and this residual, i.e., $\mathcal{L}(\delta) = \|B\delta - r^k\|_{L^2}^2$. The gradient with respect to the network parameters $\theta$ is approximated using Stratified Monte Carlo integration:
\begin{align*} \label{sec4.4:dloss}
    \nabla_\theta \mathcal{L} = \int_\Omega (B\delta-r^k)\partial_\theta B\delta\,dx
    \approx 
    \sum_{i=1}^{N_K} 2 w_i (B\delta(x_i) - r^k(x_i)) \cdot \partial_\theta (B\delta(x_i)).
\end{align*}
Assuming commutativity of $\partial_\theta$ and $B$, as $\delta \in \mathbb{U}$, it holds that $\partial_\theta\delta\in\mathbb{U}$, so we may understand $v=B(\partial_\theta\delta)\in L^2(\Omega)$ as a test function, so that the gradient of the loss acts like the residual acting on a test function via
\begin{equation}
    \nabla_\theta \mathcal{L} =  \int_\Omega (B\delta-r^k)v\,dx \approx \sum_{i=1}^{N_K} 2 w_i (B\delta(x_i) - r^k(x_i)) \cdot v(x_i).
\end{equation}
From this, we estimate the variance assuming that $v$ admits a uniform bound $|v(x)|<C$ and $w_i>0$ for any $i=1,...,N_K$. We have, via Cauchy--Schwarz
\begin{align*}
\mathbb{E}\left(\left|\sum_{i=1}^{N_K} 2 w_i (B\delta(x_i) - r^k(x_i)) \cdot v(x_i)\right|^2\right) 
&\leq 4\mathbb{E}\left[ \left(\sum_{i=1}^{N_K} w_i |B\delta(x_i) - r^k(x_i)|^2\right) \left(\sum_{i=1}^{N_K} w_i |v(x_i)|^2\right) \right] \\
&\leq 4\sup_{x \in \Omega} |v(x)|^2 \left(\sum_{i=1}^{N_K} w_i\right) \mathbb{E}\left[ \sum_{i=1}^{N_K} w_i |B\delta(x_i) - r^k(x_i)|^2 \right] \\
&= 4C^2 |\Omega| \int_\Omega |B\delta(x) - r^k(x)|^2\,dx.
\end{align*}
This quantity is an upper bound for the variance, and so we see that the variance decreases as the exact loss decreases, since $B\delta-r^k \to 0$ in the $L_2$ norm, thus, yielding a passive reduction of variance. 

\paragraph{Weak formulation (W1).}
Consider the loss function $\mathcal{L}(r) = \frac{1}{2}b(r,r) + b(u^k,r) - \ell(r)$. For a standard elliptic operator where $b(u,v)=\int \nabla u \cdot \nabla v $ and $\ell(v)=(f,v)$ the gradient with respect to a trainable parameter $\theta$ is given by the integral of:
\begin{equation}
    \nabla_\theta \mathcal{L} = b(u,v)-\ell(v) = \int_{\Omega} \left( \nabla u(x) \cdot \nabla v(x) - f(x)v(x) \right) \, dx,
\end{equation}
where $v = \partial_\theta r$ acts as a test function\footnote{It is crucial to note that the convergence of the residual $r \to 0$ does not imply that its sensitivity with respect to the parameters vanishes (i.e., $v = \partial_\theta r \not\to 0$). For instance, if $r(x; \theta) = \theta \phi(x)$ with $-\Delta \phi = f$, then at the optimum $\theta=0$, the residual vanishes but the test function $v=\phi$ does not, generating variance even at the exact solution.}. We denote the Vanilla Monte Carlo rule  for this gradient by $\hat{g} = \frac{|\Omega|}{N} \sum_{i=1}^N \nabla_\theta L(u_\theta(x_i))$. The variance of this estimator is:
\begin{equation}
    \operatorname{Var}[\hat{g}] = \frac{|\Omega|^2}{N} \operatorname{Var}[\nabla_\theta L(x)] = \frac{|\Omega|^2}{N} \left( \mathbb{E}[|\nabla_\theta L(x)|^2] - |\mathbb{E}[\nabla_\theta L(x)]|^2 \right).
\end{equation}
At the exact solution, the weak form is satisfied; thus, the expected value of the gradient is zero ($\mathbb{E}[g(x)] = 0$). Therefore, the variance reduces to the second moment of the integrand:
\begin{equation}
    \operatorname{Var}[\hat{g}] = \frac{|\Omega|}{N} \int_{\Omega} | \nabla u(x) \cdot \nabla v(x) - f(x)v(x) |^2 \, dx.
\end{equation}
Since the integrand is generally non-zero, the variance remains strictly positive ($\operatorname{Var}[\hat{g}] > 0$) and possibly large, even at the exact solution.

\paragraph{Ultra-weak formulation (U1).}
Similarly, for the ultra-weak formulation, the gradient estimator involves the adjoint operator $B'$. For the Poisson problem, the gradient with respect to a trainable parameter $\theta$ is given by:
\begin{equation}
    \nabla_\theta \mathcal{L} = (u,B'v)_{L^2}-(f,v)_{L^2} = \int_{\Omega} \left(- u(x)(\Delta v)(x) - f(x)v(x) \right) \, dx,
\end{equation}
where $v = \partial_\theta r$. The variance of the Monte Carlo estimator at the exact solution is given by:
\begin{equation}
    \operatorname{Var}[\hat{g}] = \frac{|\Omega|}{N} \int_{\Omega} | -u(x)\Delta v(x) - f(x)v(x) |^2 \, dx.
\end{equation}
Again, the condition $\nabla_\theta \mathcal{L}  = 0$ (global satisfaction of the ultra-weak form) does not enforce that the integrand vanishes pointwise. Therefore, the stochastic gradient estimator suffers from persistent variance.

\section{Numerical experiments}\label{section:Numerical_experiments}

In this section, we evaluate the performance of the Ritz--Uzawa Neural Networks (RUNNs) framework across different variational settings.  Throughout these experiments, we use the $\arg\min$ operator over the neural network parameters $\theta$ to denote the optimal parameters found during the training process. We remark that this constitutes a slight abuse of notation; due to the highly non-linear parameter space and the use of numerical optimizers, the $\arg\min$ should be understood as the practical approximation obtained after training, rather than a strict global minimum.

\subsection{Weak formulation: the Poisson problem}

We consider the domain $\Omega = (-1, 1) \subset \mathbb{R}$ and the Poisson problem with homogeneous Dirichlet boundary conditions:
\begin{equation}
\left\{
\begin{aligned}
    -u'' &= f & &\text{in } \Omega, \\
    u(-1) = u(1) &= 0, &
\end{aligned}
\right.
\label{eq:Elliptic-ode-weak}
\end{equation}
where the source term $f$ is chosen such that the exact manufactured solution is known. We employ the Sobolev spaces $\mathbb{U}=\mathbb{V}=H^1_{0}(\Omega)$, with the weak bilinear form $b(u,v)=(u',v')_{L_2(\Omega)}$ and the linear functional $\ell(v)=(f,v)_{L_2(\Omega)}$.

\paragraph{Explicit loss formulation and iterative scheme.}
Following (W1), the solution is constructed iteratively via the Uzawa scheme by minimizing a sequence of concrete energy functionals. First, the initial ansatz $u^0$ is obtained by the standard deep Ritz method:
\begin{equation}
    u^0 = \arg \min_{u_\theta} \mathcal{L}(u_\theta) := \arg \min_{u_\theta} \int_\Omega \left[ \frac{1}{2} (u_\theta'(x))^2 - f(x)u_\theta(x) \right] dx.
\end{equation}
Subsequently, for each iteration $k \ge 0$, the correction $r^k$ is computed by minimizing the residual energy, and the overall solution is updated as $u^{k+1} = u^k + r^k$ (denoted algorithmically as $u^{k+1} \leftarrow u^k + r^k$ in the subsequent figures). The functional for this correction step is given by:
\begin{equation}
    r^{k} = \arg \min_{r_\theta} \mathcal{L}_k(r_\theta) := \arg \min_{r_\theta} \int_\Omega \left[ \frac{1}{2} (r_\theta'(x))^2 + (u^k)'(x)r_\theta'(x) - f(x)r_\theta(x) \right] dx.
\end{equation}
Minimizing $\mathcal{L}_k(r_\theta)$ is equivalent to finding $\arg\min_{r_\theta} \tfrac12\| r_\theta - e^k \|_{H^1_0(\Omega)}^2$, where $e^k = u^* - u^k$ is the exact error.

\paragraph{Neural network ansatz.}
To approximate the solution, we employ a neural network ansatz $u_\theta(x)$ as defined in Section~3, explicitly selecting the smooth cut-off function $\xi(x) = (1-x^2)$ to strongly satisfy the boundary conditions by design. The specific network configuration varies depending on the complexity of the solution:
\begin{itemize}
    \item \textbf{Shallow configuration ($L=1$):} For the smooth solutions analyzed in Sections~\ref{subsec:UzawaRitz1.0} and~\ref{subsec:UzawaRitz1.1}, we use a network with $30$ neurons serving as Fourier feature generators, and no additional hidden processing layers. This acts strictly as a tunable Fourier basis.
    \item \textbf{Deep configuration ($L=2$):} For the high-frequency problem addressed in Section~\ref{subsec:UzawaRitz1.2}, we employ a network with $30$ neurons acting as Fourier feature generators, followed by one hidden layer equipped with a $\tanh$ activation function.
\end{itemize}

\paragraph{Spectral initialization strategy.}
To ensure the neural network frequencies are correctly tuned at each phase of the algorithm, we apply a consistent spectral initialization strategy across all experiments in this section:
\begin{itemize}
    \item \textbf{Initialization of $u^0$:} We analyze the source term $f$. Since $f$ naturally resides in the dual space $H^{-1}$ for a second-order elliptic problem, we compute its normalized cumulative power spectral density (NCPSD) defined in \eqref{eq: discrete_NCPSD} using a regularity index $s = -1$.
    \item \textbf{Initialization of $r^0$:} After training $u^0$, we lack a prior weak residual to analyze for the first correction. However, since $-(r^0)'' = (u^0)'' + f$ holds in the weak sense, we evaluate the strong residual $(u^0)'' + f$ and compute its NCPSD using $s = -1$.
    \item \textbf{Initialization of $r^k$ ($k \ge 1$):} For subsequent Uzawa iterative steps, we analyze the previously trained weak residual representative $r^{k-1}$. Because $r^{k-1}$ belongs to the test space $\mathbb{V} = H^1_0$ and approximates the error in the energy norm, we compute its NCPSD using $s = 1$.
\end{itemize}

We employ a high-density $P_3$ quadrature rule (as detailed in Section~\ref{sec:Loss_Discretization}) across all experiments to accurately evaluate these distributions and compute the loss functions. The specific training schedules (epochs, collocation points, and learning rates) are detailed in the respective tables for each experiment.

\subsubsection{Baseline with Adam optimization}\label{subsec:UzawaRitz1.0}

Our initial experiment targets the smooth manufactured solution $u^{*}(x)=\sin(\pi x)$. We employ the standard Adam optimizer alongside the Shallow Configuration to train both the initial guess $u^0$ and the subsequent iterative corrections $r^k$, adhering to the training schedule detailed in Table~\ref{tab:adam_regime_params}. The evolution of the relative error and the corresponding NCPSD analysis to select $[\omega_{\min},\omega_{\max}]$ are captured in Figure~\ref{fig:Combined_Uzawa_results_1.0}. Providing deeper insight into the correction mechanism, Figure~\ref{fig:correction_analysis_1.0} demonstrates that the method successfully approximates the exact error function $e^0$ and its derivative. However, while the approach ultimately yields a highly accurate final error function (Figure~\ref{fig:final_error_1.0}), this pure Adam optimization strategy exhibits noticeably slower convergence when compared to the hybrid methods explored in subsequent sections.

\begin{table}[H]
\centering
\begin{tabular}{lcccc}
\hline
\textbf{Iteration} & \textbf{Phase} & \textbf{Points ($N_K$)} & \textbf{Epochs} & \textbf{Learning Rate} \\ \hline
$k=0$ (Initial) & $u^0$ & $9000$ & $1000$ & $9\cdot 10^{-3}$ \\
$k=1$ & $r^0$ & $9000$ & $2000$ & $10^{-4}$ \\
$k=2$ & $r^1$ & $9000$ & $3000$ & $10^{-5}$ \\
\end{tabular}
\caption{Training hyperparameters for the pure Adam experiment.}
\label{tab:adam_regime_params}
\end{table}

\begin{figure}[H]
    \centering
    \subfloat[Relative error evolution ($6,000$ epochs).]{%
        \label{fig:Uzawa_weak_error_1.0}%
        \includegraphics[width=0.48\textwidth]{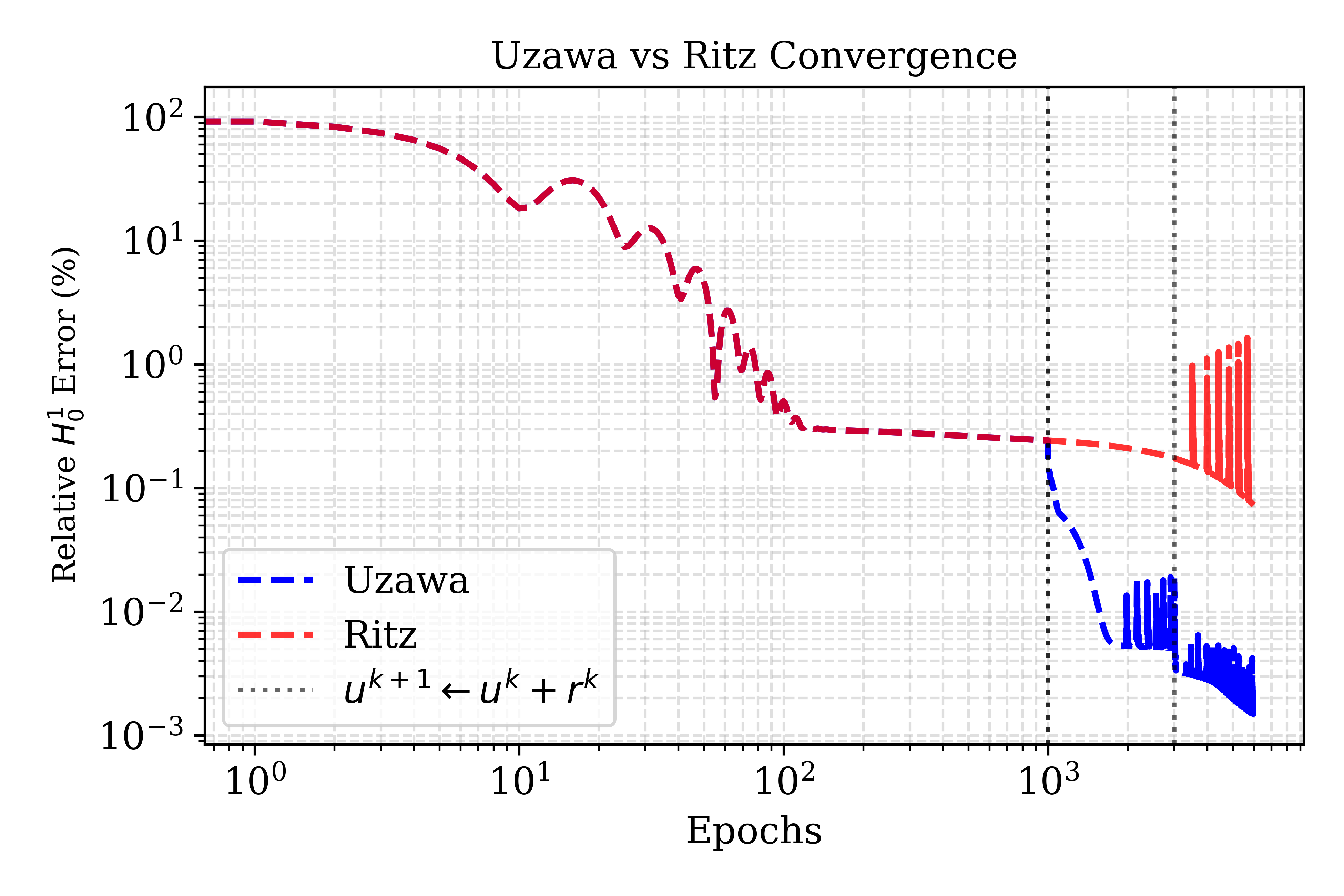}%
    }
    \hfill
    \subfloat[NCPSD bandwidth selection.]{%
        \label{fig:Uzawa_weak_variance_1.0}%
        \includegraphics[width=0.48\textwidth]{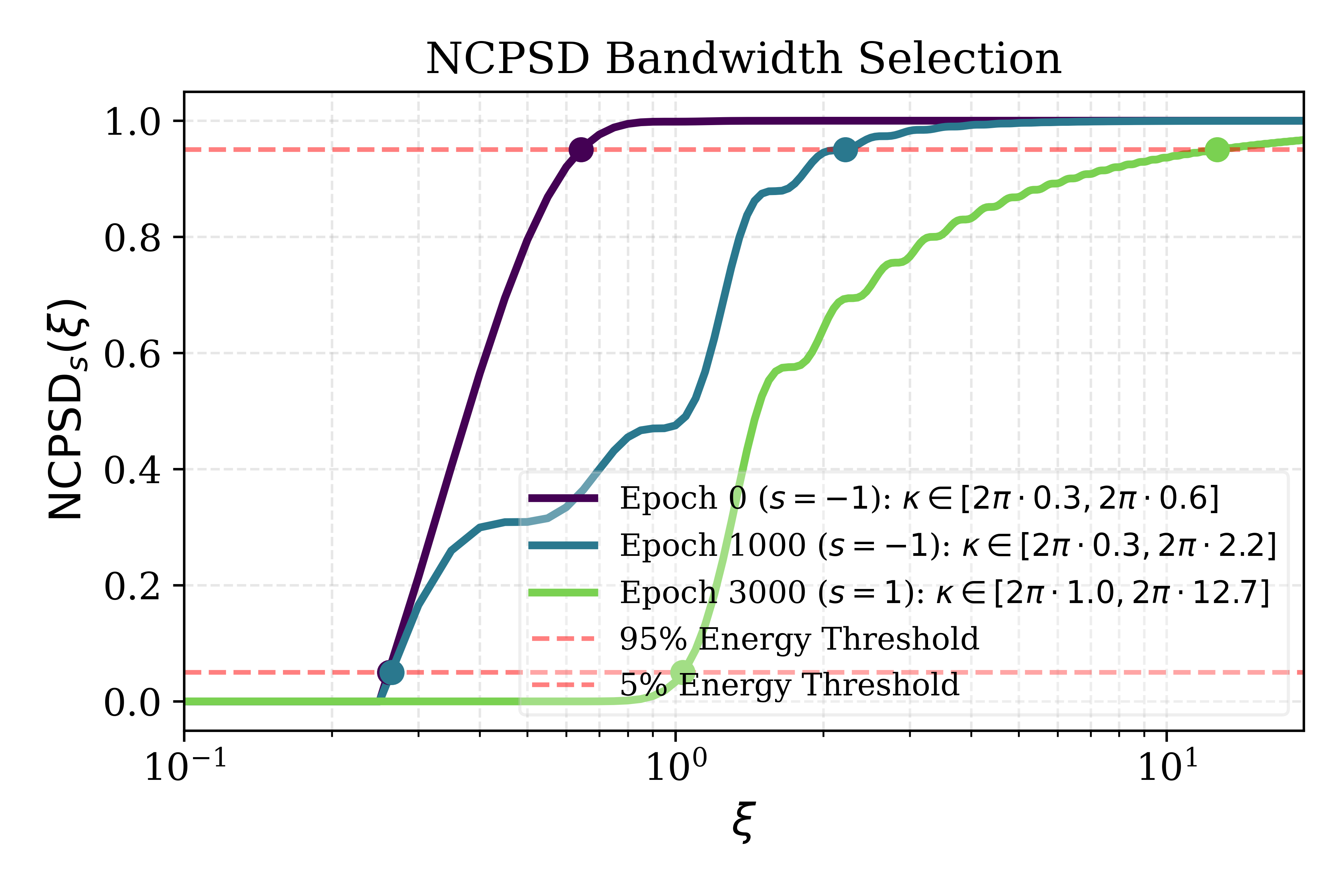}%
    }
    \caption{Experimental results with Adam: (a) Relative error and (b) The NCPSD analysis.}
    \label{fig:Combined_Uzawa_results_1.0}
\end{figure}

\begin{figure}[H]
    \centering
    \includegraphics[width=0.8\textwidth]{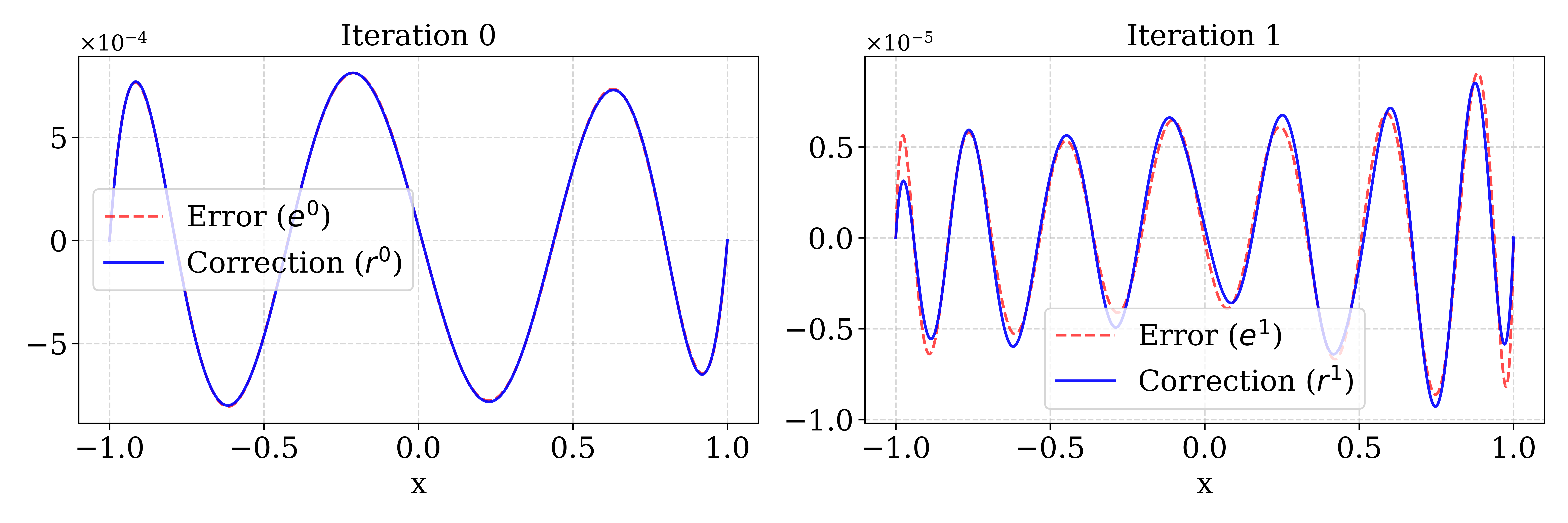}
    \vspace{0.5em}
    \includegraphics[width=0.8\textwidth]{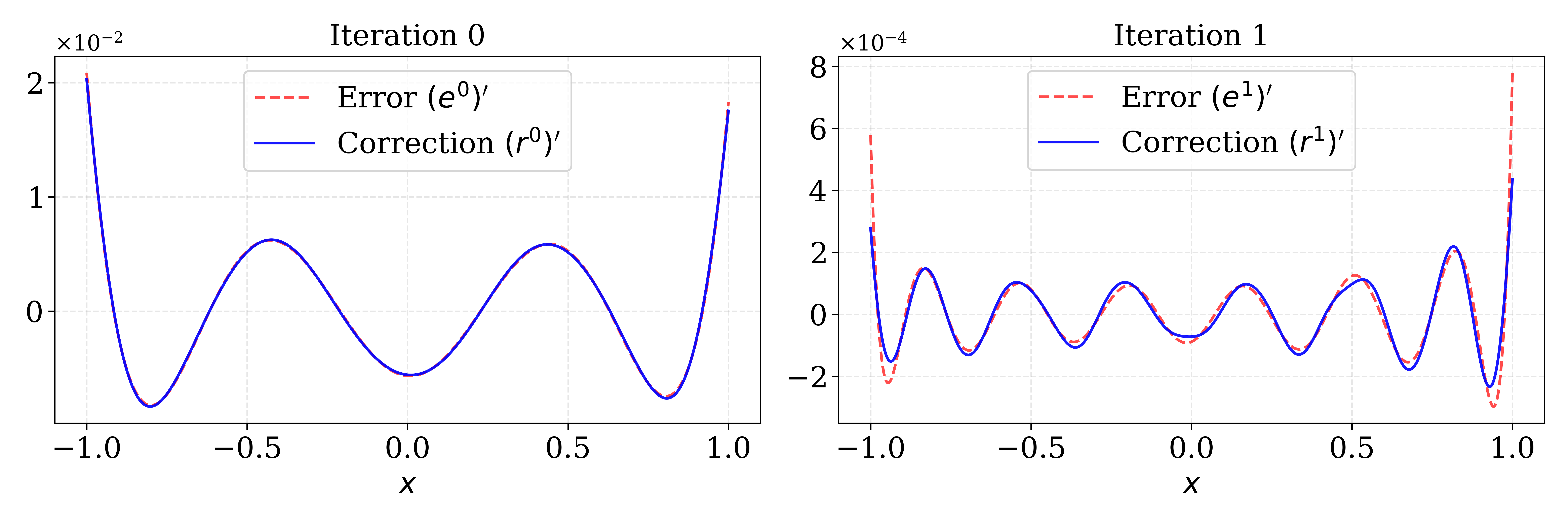}
    \caption{Analysis of the correction step ($k=0$) after $6,000$ epochs. Top: Comparison between the true error $e^0$ and the learned correction $r^0$. Bottom: Comparison of their derivatives.}
    \label{fig:correction_analysis_1.0}
\end{figure}

\begin{figure}[H]
    \centering
    \includegraphics[width=0.8\textwidth]{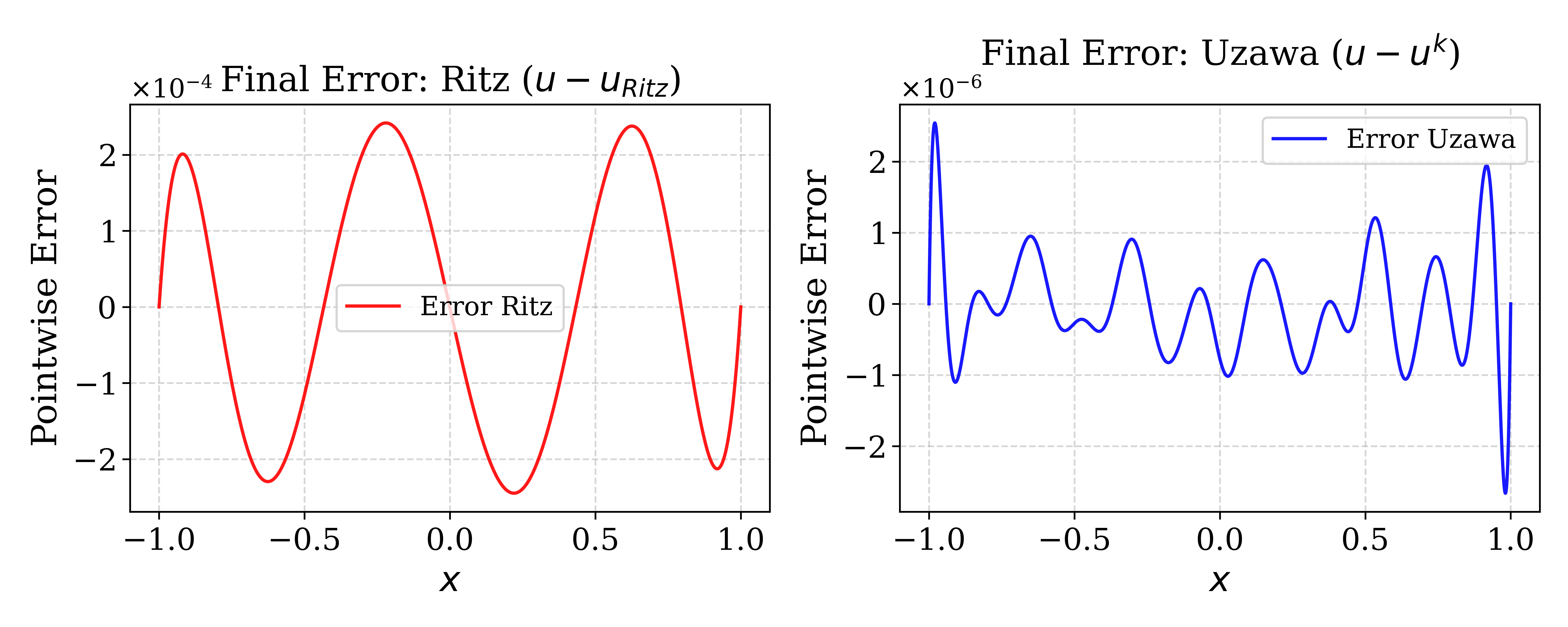}
    \caption{Final error function.}
    \label{fig:final_error_1.0}
\end{figure}

\subsubsection{Acceleration with hybrid LS/Adam}\label{subsec:UzawaRitz1.1}

To accelerate convergence, we repeat the baseline experiment using the Hybrid LS/Adam optimizer, maintaining the Shallow Configuration for both the initial guess and its iterative corrections. By alternating between least-squares projections for the linear weights and Adam updates for the hidden parameters (detailed in Table~\ref{tab:ls_adam_regime_params}), the method achieves a rapid reduction in the relative $H^1_0$ norm (Figure~\ref{fig:Uzawa_weak_error_1.1}). Concurrently, the Uzawa-based spectral matching successfully isolates higher-frequency modes (Figure~\ref{fig:Uzawa_ncpsd_1.1}). The efficacy of this correction mechanism is validated in Figure~\ref{fig:correction_analysis_1.1}, which demonstrates a strong agreement between the residual $r^k$ and the exact error $e^k$ alongside their derivatives, ultimately resulting in the highly accurate final error function shown in Figure~\ref{fig:final_error_1.1}.
\begin{table}[H]
\centering
\begin{tabular}{lcccc}
\hline
\textbf{Iteration} & \textbf{Phase} & \textbf{Points ($N_K$)} & \textbf{Epochs} & \textbf{Learning Rate} \\ \hline
$k=0$ (Initial) & $u^0$ & $9000$ & $1000$ & $10^{-2}$ \\
$k=1$ & $r^0$ & $9000$ & $1000$ & $10^{-3}$ \\
$k=2$ & $r^1$ & $9000$ & $1000$ & $10^{-3}$ \\
\end{tabular}
\caption{Training hyperparameters (Hybrid LS/Adam). In this setup, both the Gradient step (Adam) and the least squares step utilize the same density of collocation points.}
\label{tab:ls_adam_regime_params}
\end{table}

\begin{figure}[H]
    \centering
    \subfloat[Relative error evolution ($3,000$ total epochs).]{%
        \label{fig:Uzawa_weak_error_1.1}%
        \includegraphics[width=0.48\textwidth]{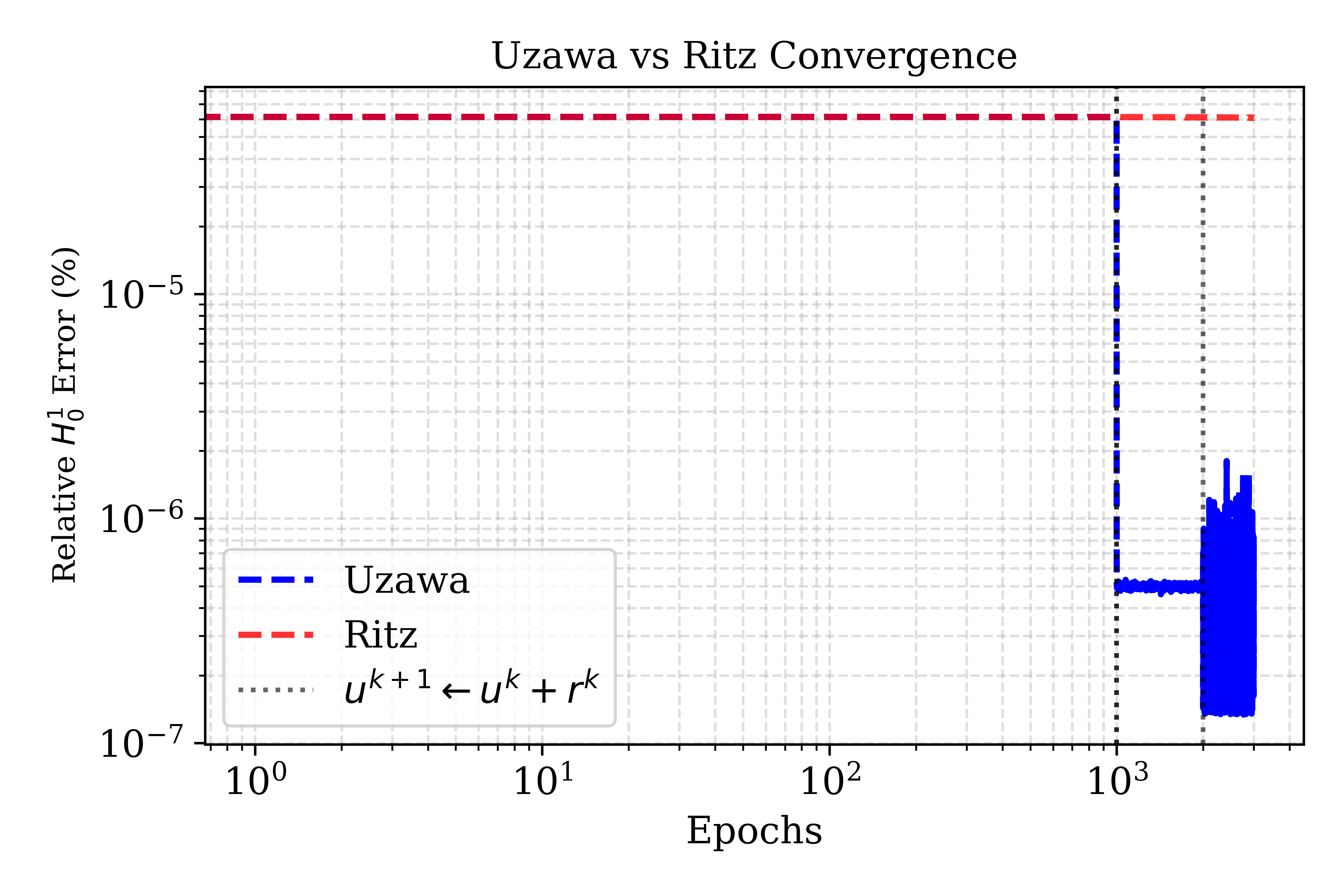}%
    }
    \hfill
    \subfloat[NCPSD bandwidth selection.]{%
        \label{fig:Uzawa_ncpsd_1.1}%
        \includegraphics[width=0.48\textwidth]{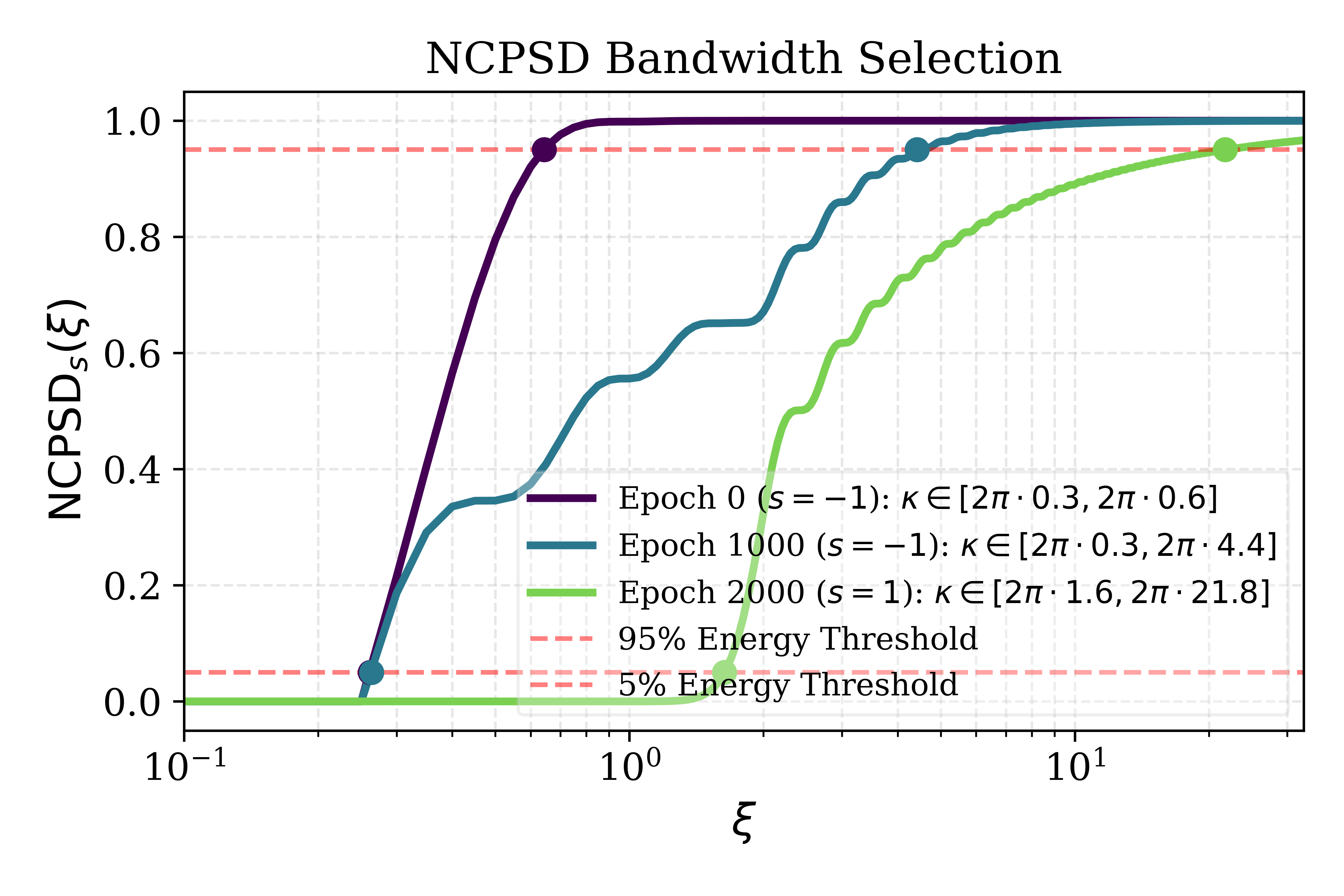}%
    }
    \caption{Experimental results with LS/Adam: (a) Relative error and (b) The NCPSD analysis.}
    \label{fig:Combined_Uzawa_results_1.1}
\end{figure}

As shown in Figure~\ref{fig:Combined_Uzawa_results_1.1}, the LS/Adam strategy drastically reduces the training epochs, achieving superior accuracy in $3,000$ total epochs compared to the Adam optimization.

\begin{figure}[H]
    \centering
    \includegraphics[width=0.75\textwidth]{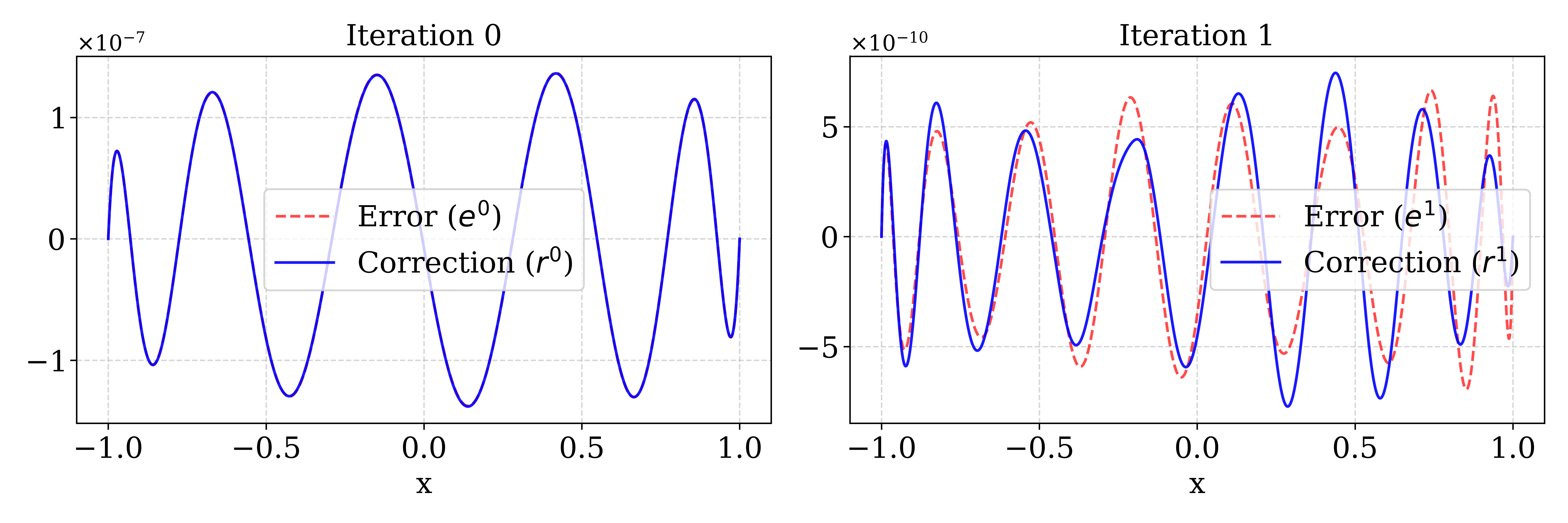}
    \vspace{0.5em}
    \includegraphics[width=0.75\textwidth]{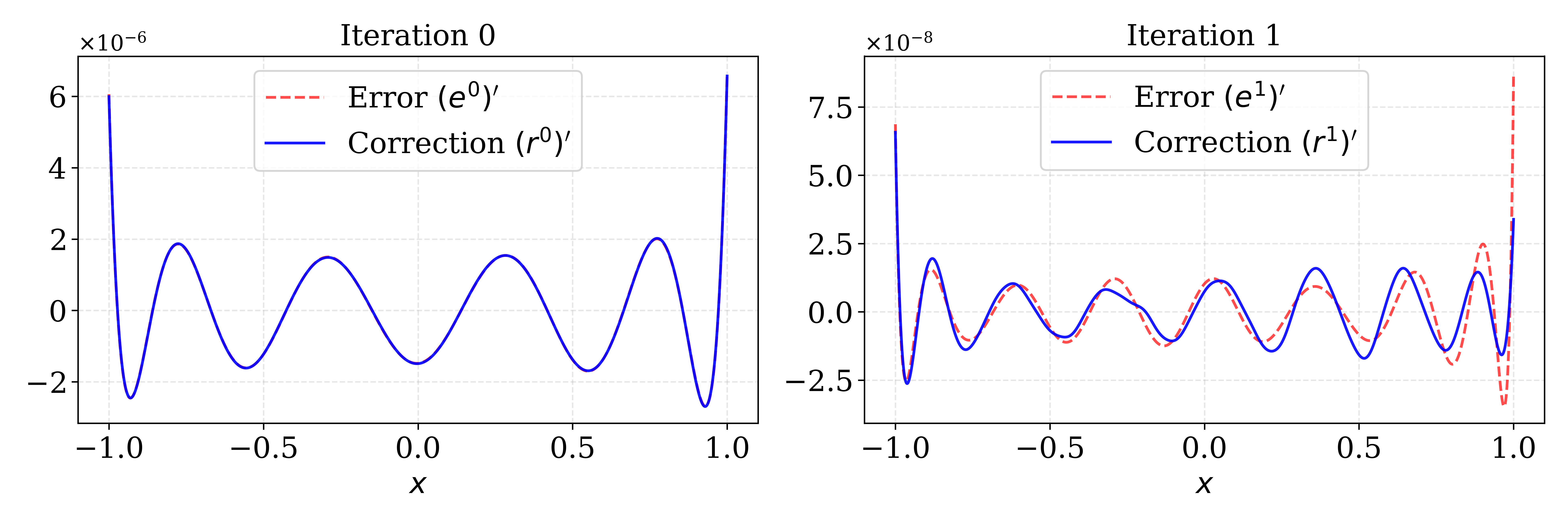}
    \caption{Analysis of the correction step ($k=0$) after $3,000$ epochs. Top: Comparison between the true error $e^0$ and the learned correction $r^0$. Bottom: Comparison of their derivatives.}
    \label{fig:correction_analysis_1.1}
\end{figure}

\begin{figure}[H]
    \centering
    \includegraphics[width=0.8\textwidth]{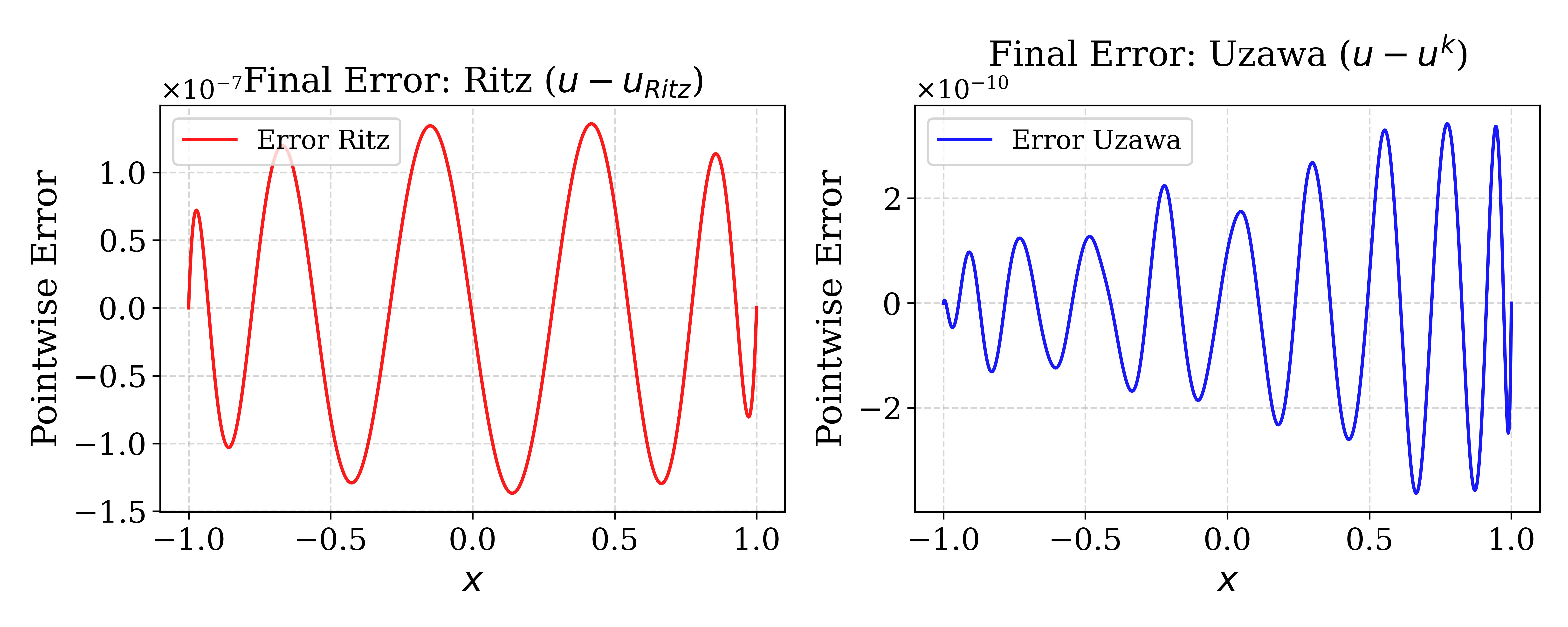}
    \caption{Final error function.}
    \label{fig:final_error_1.1}
\end{figure}

\subsubsection{High-frequency resolution via spectral matching}\label{subsec:UzawaRitz1.2}

In this experiment, we evaluate the method using the high-frequency manufactured solution $u^*(x) = \sin(40\pi x)$ to address spectral bias. The model is trained using the hybrid LS/Adam optimizer with the hyperparameters detailed in Table~\ref{tab:ls_adam_regime_params_1.2}.  Figure~\ref{fig:Combined_Uzawa_results_1.2} shows how we select the bandwidth $[\omega_{\min},\omega_{\max}]$. Furthermore, Figure~\ref{fig:correction_analysis_1.2} shows that the hybrid architecture captures the high-frequency oscillations of the error, resulting in the final pointwise error presented in Figure~\ref{fig:final_error_1.2}.

\begin{table}[H]
\centering
\begin{tabular}{lcccc}
\hline
\textbf{Iteration} & \textbf{Phase} & \textbf{Points ($N_K$)} & \textbf{Epochs} & \textbf{Learning Rate} \\ \hline
$k=0$ (Initial) & $u^0$ & $9000$ & $300$ & $10^{-3}$ \\
$k=1$ & $r^0$ & $9000$ & $1000$ & $10^{-3}$ \\
$k=2$ & $r^1$ & $9000$ & $1000$ & $10^{-3}$ \\
\end{tabular}
\caption{Training hyperparameters.}
\label{tab:ls_adam_regime_params_1.2}
\end{table}

\begin{figure}[H]
    \centering
    \begin{subfigure}[b]{0.48\textwidth}
        \includegraphics[width=\textwidth]{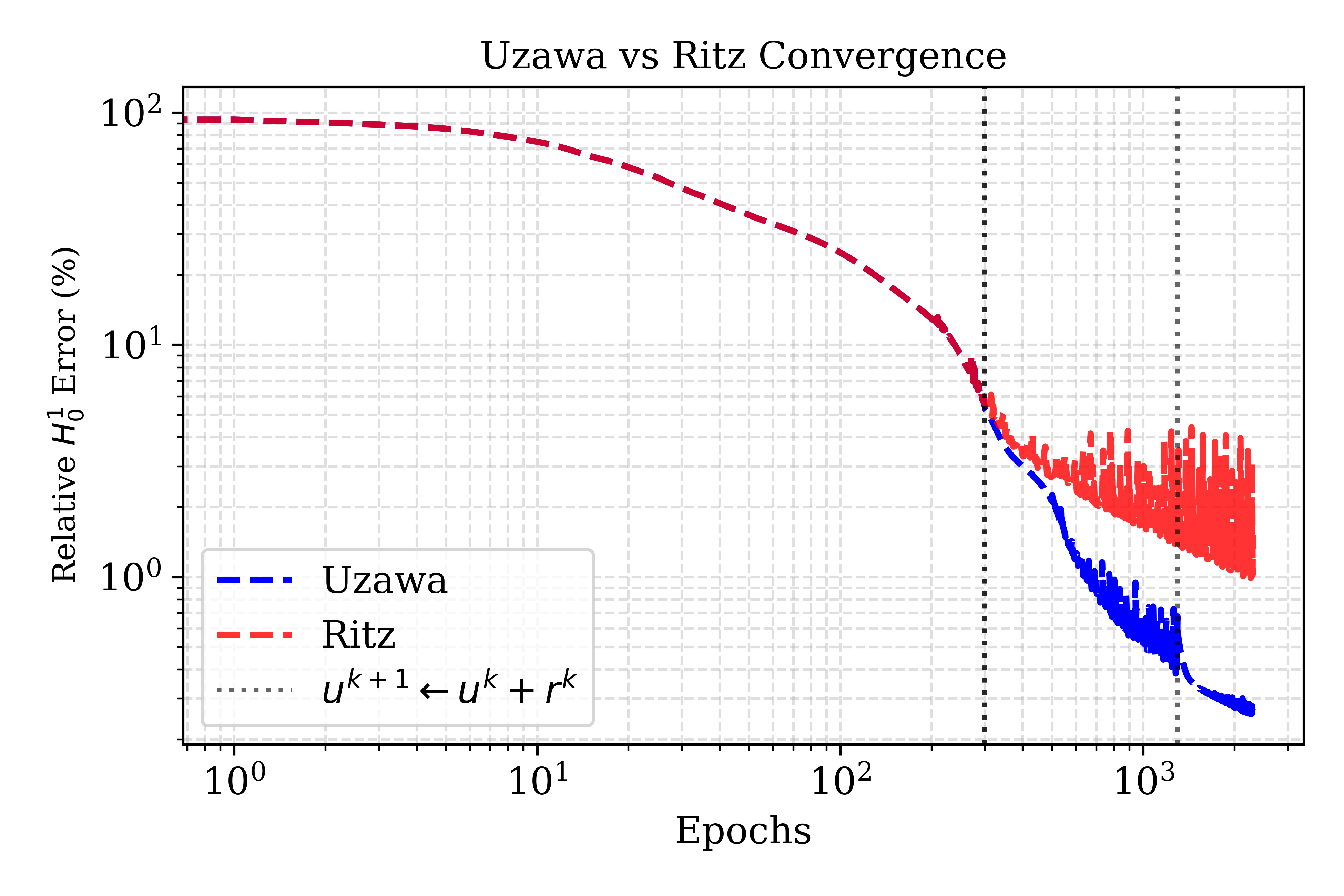}
        \caption{Relative error evolution.}
        \label{fig:Uzawa_weak_error_1.2}
    \end{subfigure}
    \hfill
    \begin{subfigure}[b]{0.48\textwidth}
        \includegraphics[width=\textwidth]{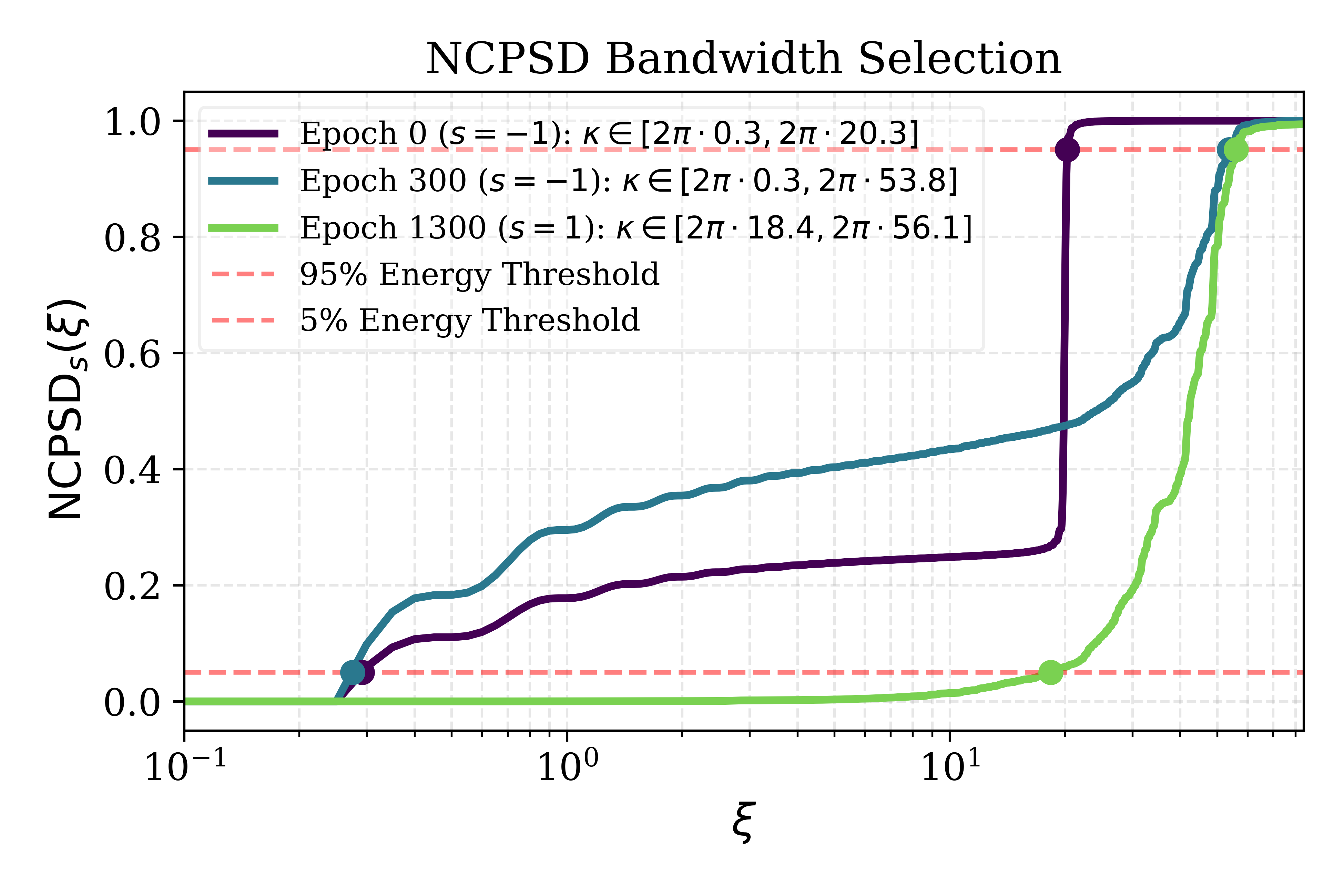}
        \caption{NCPSD bandwidth selection.}
        \label{fig:spectral_ncpsd_1.2}
    \end{subfigure}
    \caption{Experimental results for the high-frequency case: (a) Relative error. (b) The NCPSD bandwidth selection.}
    \label{fig:Combined_Uzawa_results_1.2}
\end{figure}

\begin{figure}[H]
    \centering
    \includegraphics[width=0.75\textwidth]{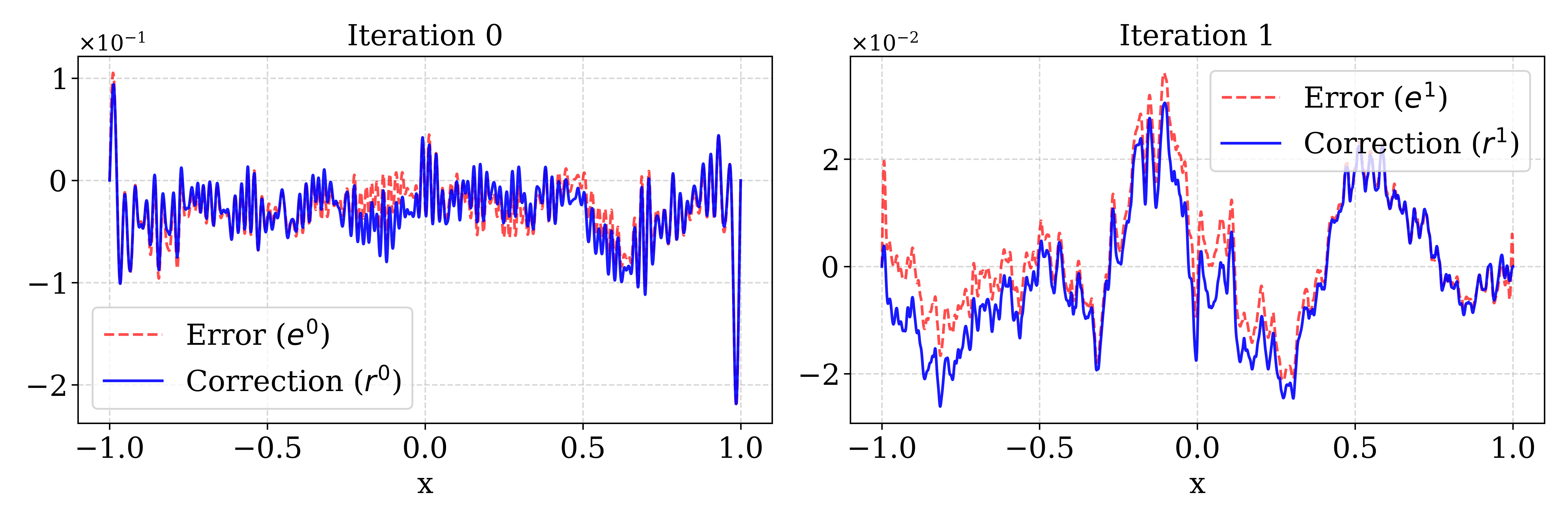}
    \vspace{0.5em}
    \includegraphics[width=0.75\textwidth]{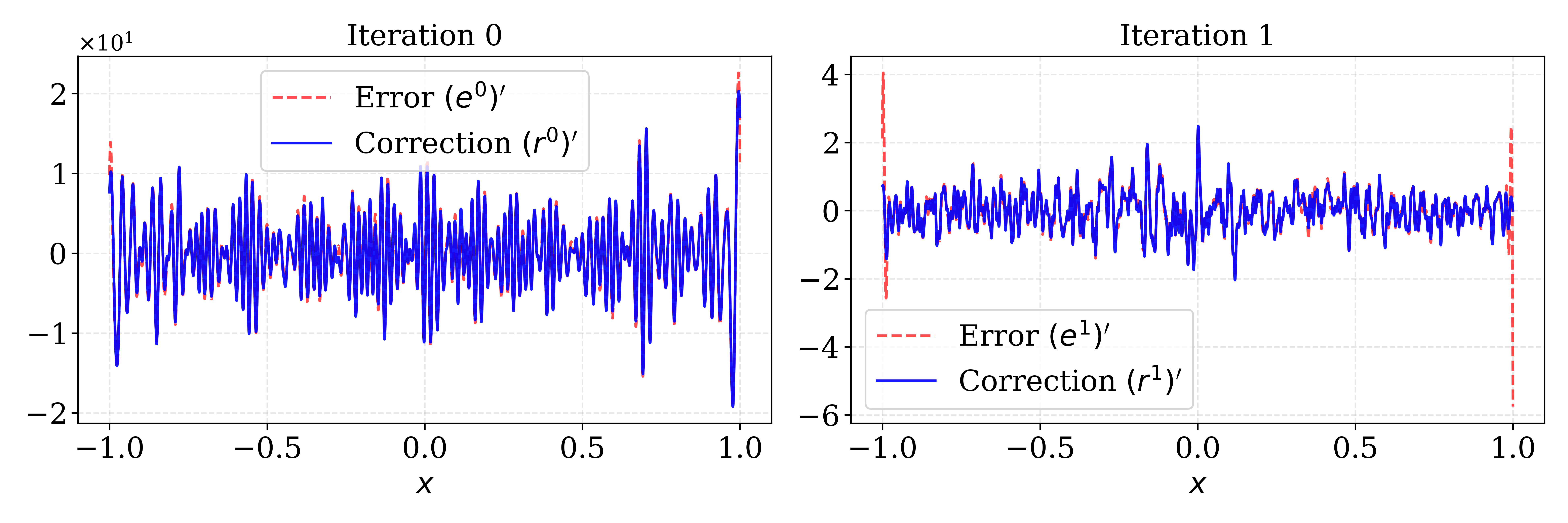}
    \caption{Detailed correction analysis for $u^*(x) = \sin(40\pi x)$. The hybrid architecture captures the high-frequency oscillations of the error $e^0$.}
    \label{fig:correction_analysis_1.2}
\end{figure}

\begin{figure}[H]
    \centering
    \includegraphics[width=0.8\textwidth]{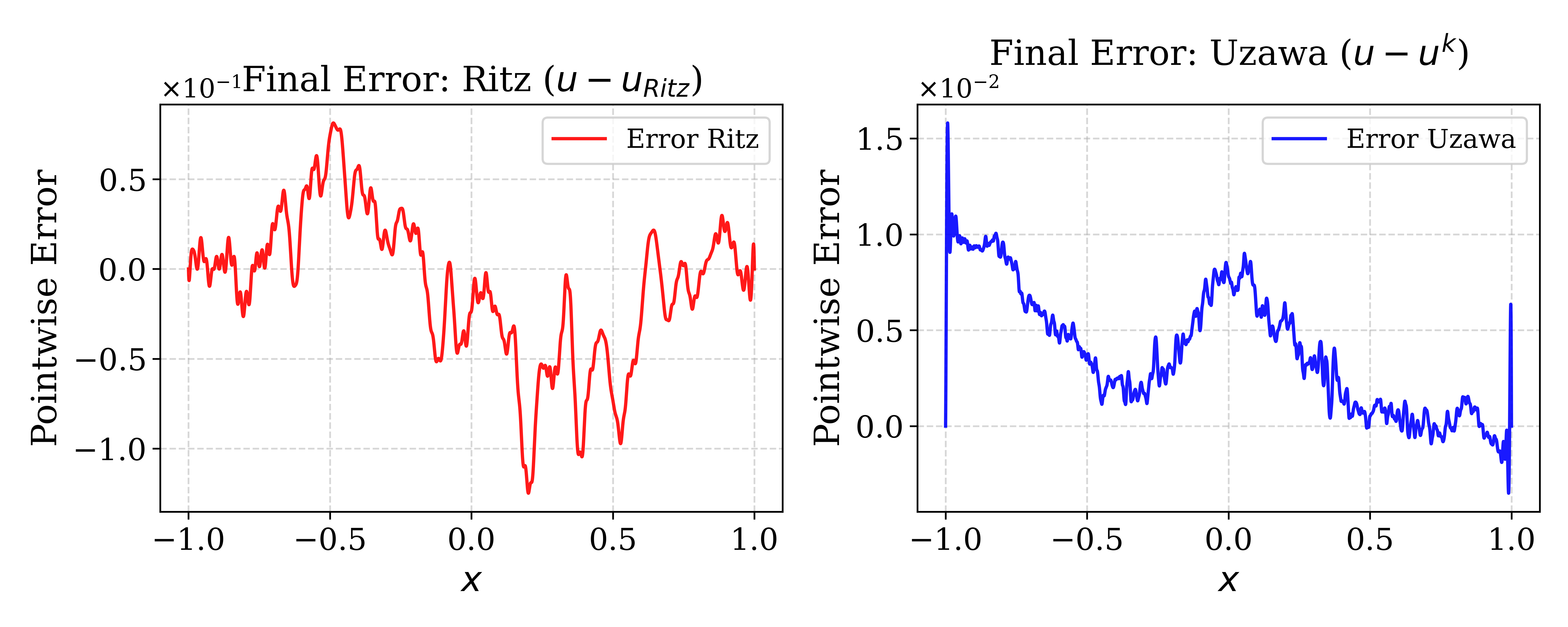}
    \caption{Final pointwise error for the high-frequency experiment.}
    \label{fig:final_error_1.2}
\end{figure}

\subsection{Ultra-weak formulation with an $H^{-2}$ source term}

We consider the domain $\Omega = (-1, 1) \subset \mathbb{R}$ and the boundary value problem governed by a distributional source term:
\begin{equation}
\left\{
\begin{aligned}
    -u'' &= \delta' & &\text{in } \Omega, \\
    u(-1) = u(1) &= 0. &
\end{aligned}
\right.
\label{eq:Elliptic-ode-ultraweak}
\end{equation}

Here, $\delta'$ denotes the distributional derivative of the Dirac delta. The exact solution is $u^{*}(x)=\tfrac12 (x+1)-H(x)$, where $H(x)$ is the Heaviside step function. We follow (U1) with $\mathbb{U}=L_2(\Omega)$ and $\mathbb{V}=H^1_{0}(\Omega) \cap H^2(\Omega)$.
\paragraph{Explicit loss formulation.}
Following (U1), we minimize the energy of the test functions to iteratively recover the primal solution, observing that the test-to-trial operator in this setting acts as $B'v = -v''$. First, the initial test function approximation $v$ is obtained by minimizing the Adjoint Ritz functional:
\begin{equation}
    v = \arg \min_{v_\theta} \int_\Omega \frac{1}{2} (v_\theta''(x))^2 dx + v_\theta'(0).
\end{equation}
The initial primal ansatz is then recovered via $u^0 = B'v = -(v)''$. Subsequently, for each iteration $k \ge 0$, the correction test function $r^k$ is computed by minimizing the Uzawa loss, which explicitly incorporates the previous primal approximation $u^k$:
\begin{equation}
    r^{k} = \arg \min_{r_\theta} \int_\Omega \left[ \frac{1}{2} (r_\theta''(x))^2 + u^k(x) r_\theta''(x) \right] dx + r_\theta'(0).
\end{equation}
Note that the isolated derivative terms ($v_\theta'(0)$ and $r_\theta'(0)$) arise from the duality pairing $\langle \delta', \phi \rangle := -\phi'(0)$, which appears with a negative sign ($-\ell(\phi)$) in the respective energy functionals. Finally, following the Uzawa scheme, the primal solution is updated as $u^{k+1} = u^k + B'r^k = u^k - (r^k)''$. In the algorithmic context of the figures, this step is denoted as $u^{k+1} \leftarrow u^k + \delta^k$, where $\delta^k = -(r^k)''$ represents the effective correction in the trial space.

\paragraph{Neural network ansatz.}
To ensure that the neural network output is $H^2$-continuous (as strictly required by the test space), the initial approximation $u^0 = -(v)''$ is calculated using a shallow neural network $v_\theta$ equipped with cubic ReLU activations ($\text{ReLU}^3$). For subsequent Uzawa iterations, the corrections $r^k$ are calculated using a neural network composed of two layers: the first layer is a Fourier feature mapping and the second layer is a $\text{ReLU}^3$ network with 30 neurons.

\paragraph{Spectral initialization strategy.}
For the distributional source $f = \delta' \in H^{-2}$, the $H^{-2}$-weighted energy spectrum decays slowly, reflecting a heavy high-frequency tail. This severe lack of regularity makes it difficult for standard networks to capture the jump discontinuity. We apply the following spectral initialization strategy:
\begin{itemize}
    \item \textbf{Initialization of $r^0$:} To tune the frequencies for the first correction, we evaluate the highly singular strong residual $-(u^0)'' - f$ and compute its NCPSD, also utilizing the regularity index $s = -2$.
    \item \textbf{Initialization of $r^k$ ($k \ge 1$):} For subsequent Uzawa iterative steps, we rely on the previously trained adjoint test functions. Specifically, we compute the NCPSD directly on the previous primal correction proxy $-(r^{k-1})''$. To capture its spectral energy distribution effectively without over-weighting the singularities, we evaluate the NCPSD in the $L^2$ norm by setting $s = 0$.
\end{itemize}

\subsubsection{Singularity resolution via adjoint iterations}
The model is trained using the hybrid LS/Adam optimizer with the hyperparameters detailed in Table~\ref{tab:ls_adam_regime_params_1.3}. Figure~\ref{fig:Uzawa_ultraweak_error_1.3} illustrates the relative error in the $L^2$ norm, and the bandwidth selection $[\omega_{\min},\omega_{\max}]$ in the NCPSD analysis is shown in Figure~\ref{fig:Uzawa_ultraweak_ncpsd_1.3}. The correction mechanism is analyzed in Figure~\ref{fig:correction_analysis_1.3}, showing that the network identifies the error distribution $e^k$ through the learned correction $-(r^k)''$, capturing the jump discontinuity at $x=0$. The final pointwise error is presented in Figure~\ref{fig:error_ek_rk_1.3}.

\begin{table}[H]
\centering
\begin{tabular}{lcccc}
\hline
\textbf{Iteration} & \textbf{Phase} & \textbf{Points ($N_K$)} & \textbf{Epochs} & \textbf{Learning Rate} \\ \hline
$k=0$ (Init) & $u^0$ & $3000$ & $100$ & $8\cdot 10^{-5}$ \\
$k=1$ & $r^0$ & $4500$ & $2500$ & $8\cdot 10^{-5}$ \\
$k=2$ & $r^1$ & $6000$ & $2500$ & $8\cdot 10^{-5}$ \\
\end{tabular}
\caption{Training hyperparameters for the ultra-weak experiment.}
\label{tab:ls_adam_regime_params_1.3}
\end{table}

\begin{figure}[H]
    \centering
    \subfloat[Relative error evolution ($5,100$ total epochs).]{%
        \label{fig:Uzawa_ultraweak_error_1.3}%
        \includegraphics[width=0.48\textwidth]{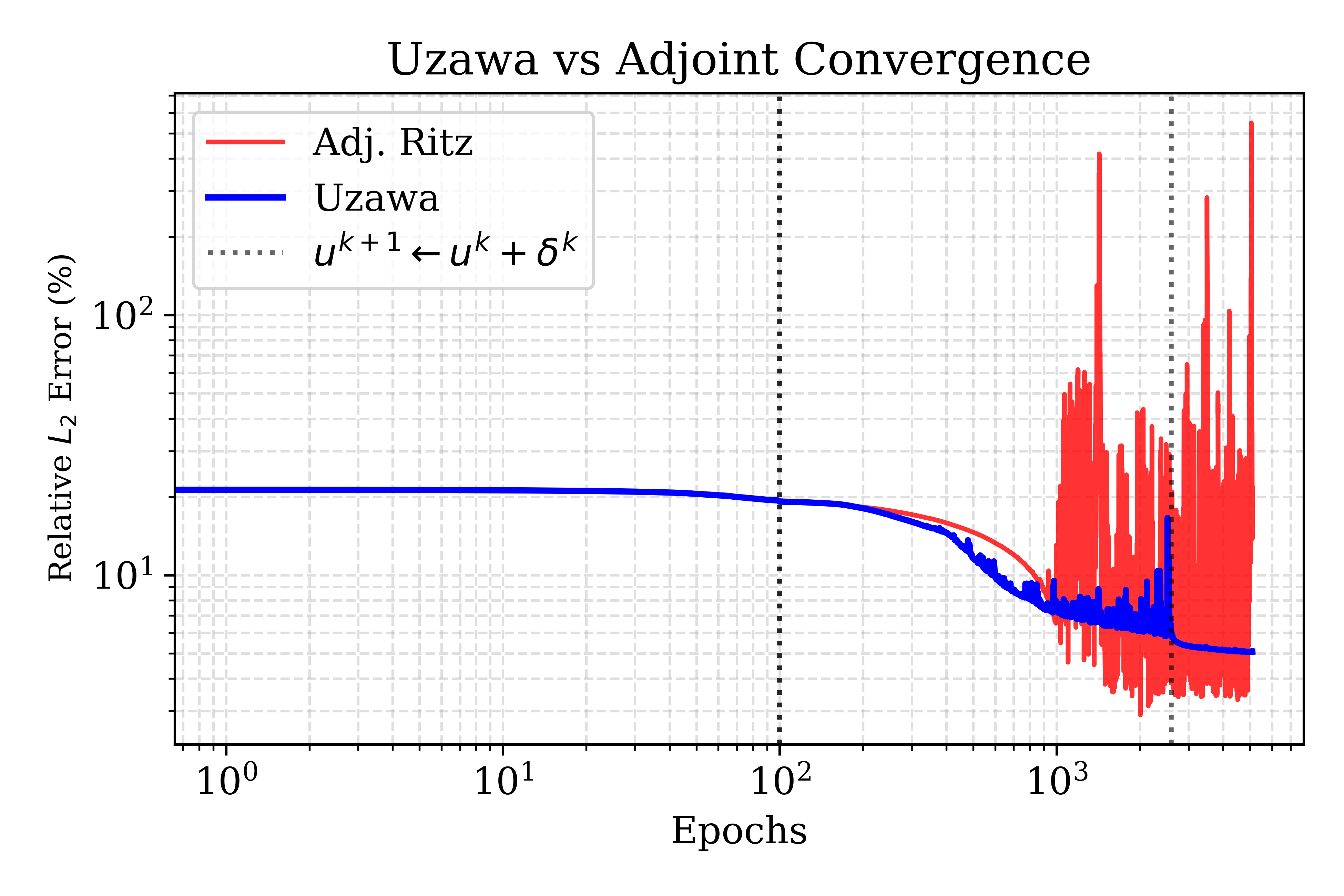}%
    }
    \hfill
    \subfloat[NCPSD bandwidth selection.]{%
        \label{fig:Uzawa_ultraweak_ncpsd_1.3}%
        \includegraphics[width=0.48\textwidth]{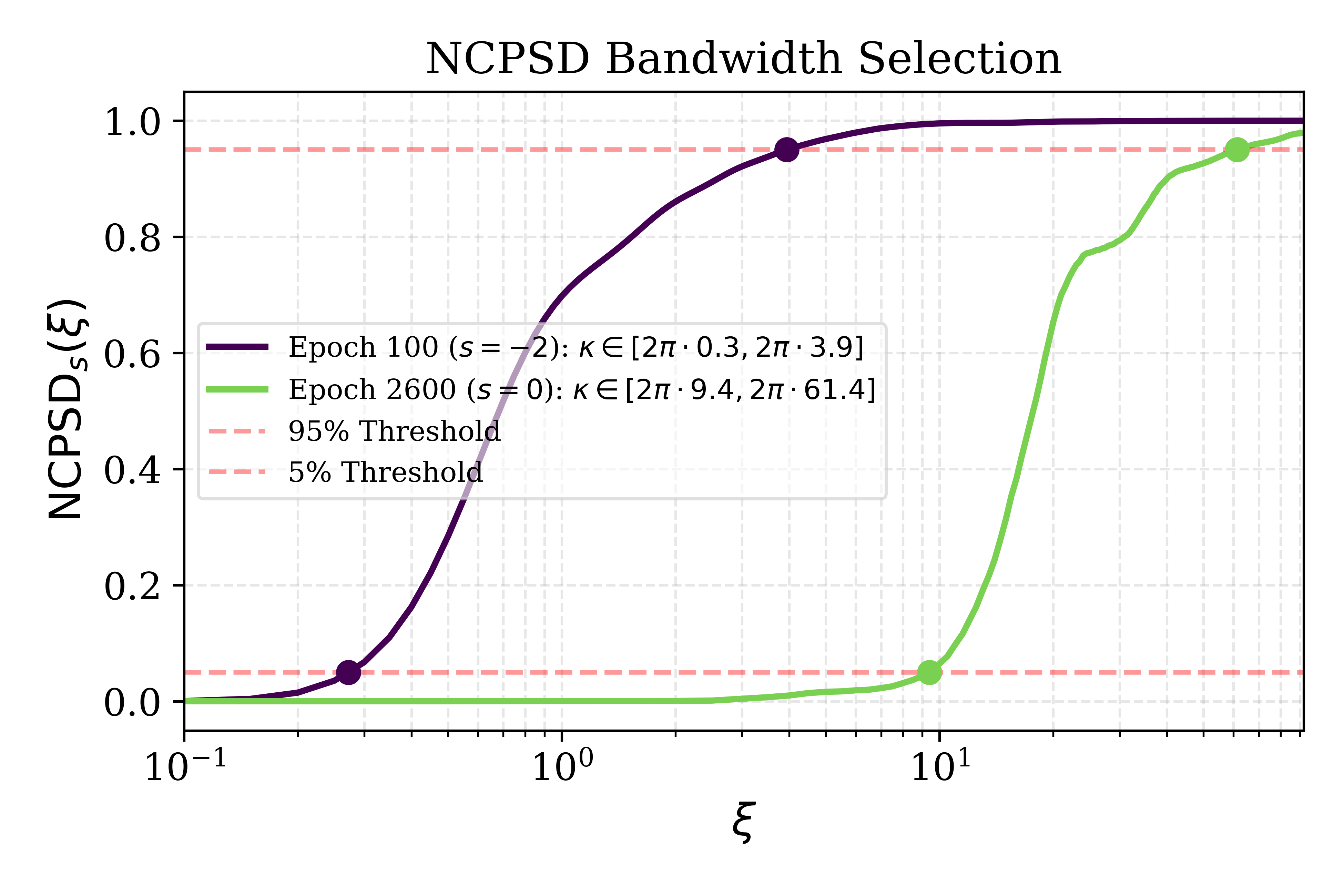}%
    }
    \caption{Experimental results for the ultra-weak formulation: (a) Relative error evolution and (b) Evolution of the NCPSD.}
    \label{fig:Combined_Uzawa_results_1.3}
\end{figure}

\begin{figure}[H]
    \centering
    \includegraphics[width=0.8\textwidth]{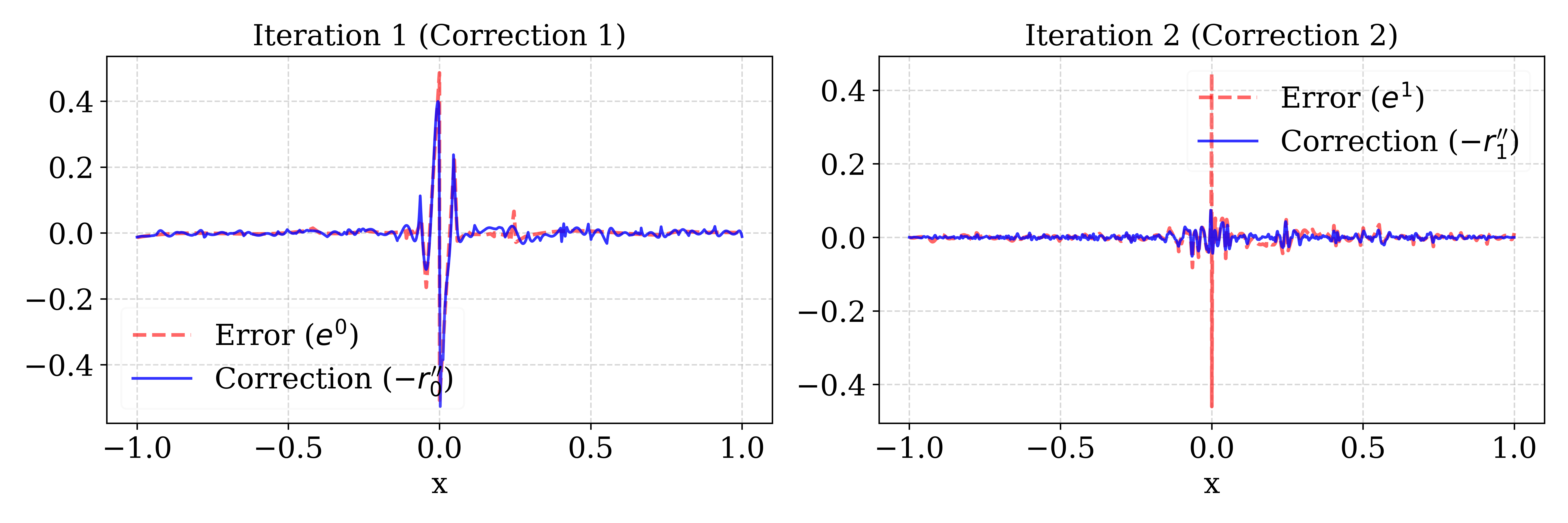}
    \caption{Correction step. The plot compares the true error $e^{k}$ (red dashed line) with the learned correction $-(r^k)''$ (blue solid line).}   
    \label{fig:correction_analysis_1.3}
\end{figure}

\begin{figure}[H]
    \centering
    \includegraphics[width=0.8\textwidth]{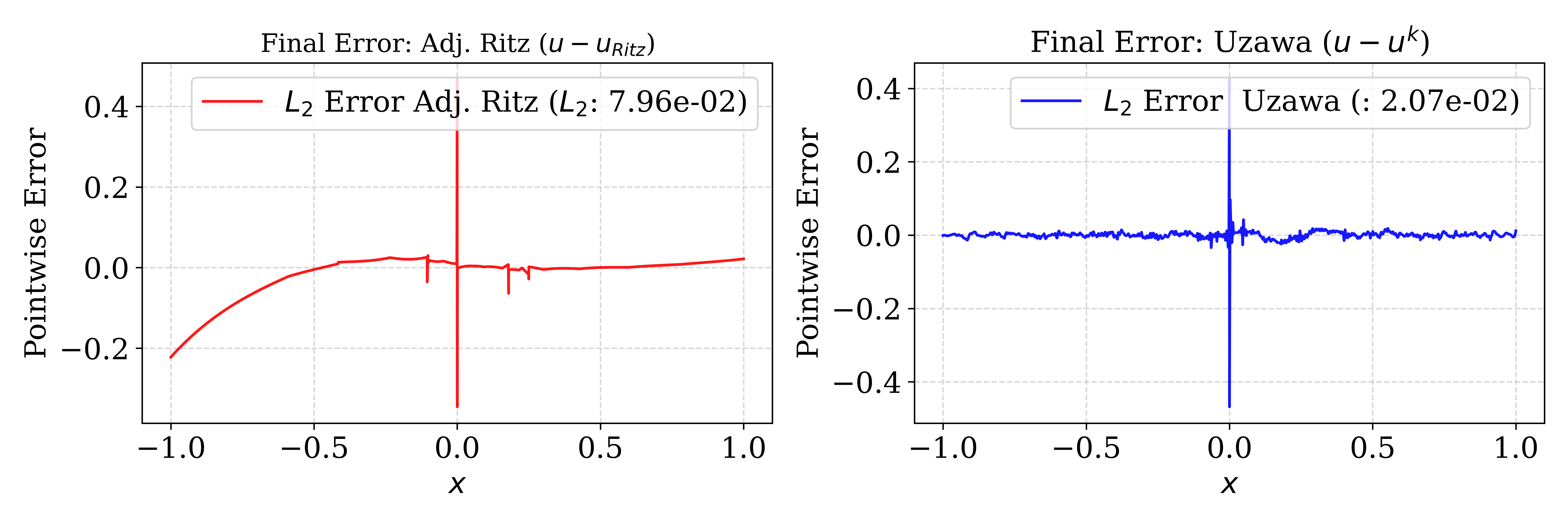}
    \caption{Final pointwise error.}
    \label{fig:error_ek_rk_1.3}
\end{figure}

\section{Conclusions and future work}

In this paper, we introduced the Ritz--Uzawa Neural Networks (RUNNs) framework, an iterative methodology to solve Partial Differential Equations (PDEs) across strong, weak, and ultra-weak variational formulations. By reformulating the variational problem as a sequence of Ritz-type minimizations within an inexact Uzawa loop, our approach addresses the instability common in standard neural network solvers. Theoretical analysis and numerical experiments confirm that this iterative scheme, combined with a high-density, unbiased stratified stochastic quadrature rule ($P_3$), converges to the solution. Furthermore, we show that the strong formulation provides a passive variance reduction mechanism, while the variance remains persistent in weak and ultra-weak regimes.

To address the spectral bias of standard neural networks, we implemented a data-driven frequency tuning strategy. By dynamically initializing a Sinusoidal Fourier Feature Mapping based on the normalized cumulative power spectral density (NCPSD) of the residuals, the network adapts its spectral bandwidth to capture high-frequency components and severe singularities. Combined with a hybrid least Squares/Adam (LS/Adam) optimization scheme, the RUNNs framework significantly accelerates convergence compared to first order gradient descent methods like Adam.

Numerical results demonstrate the robustness of RUNNs across different regularity regimes. The framework accurately resolved highly oscillatory solutions via spectral matching and successfully recovered a discontinuous $L^2$ solution from a distributional $H^{-2}$ source term using the ultra-weak formulation, a scenario where standard energy-based methods fail. These findings establish the Ritz--Uzawa iterative scheme as a stable and adaptable mesh-free solver. Future work will extend this framework to non-linear PDEs, higher-dimensional geometries, and time-dependent equations, where dynamic spectral tuning and variance-reduction iterations could offer substantial computational advantages.

\paragraph{Acknowledgments}
Pablo Herrera has received financial support from the Dirección de Investigación (DI), under the Vicerrectoría de Investigación, Creación e Innovación (VINCI) of the Pontificia Universidad Católica de Valparaíso (PUCV), through the 2024 Postdoctoral Fellowship. Jamie M. Taylor has received funding from the research project: PID2023-146678OB-I00 funded by MICIU/AEI /10.13039/501100011033. Carlos Uriarte is supported by the Basque Government through the Postdoctoral Program for the Improvement of Doctoral Research Personnel for the period 2024-2027 (Grant Ref. No. POS-2024-1-0004) and by the Research Projects PID2023-146678OB-I00 and PID2023-146668OA-I00, both funded by MICIU/AEI /10.13039/501100011033. Ignacio Muga is supported by the Chilean National Agency for Research and Development through the Fondecyt Project \#1230091. David Pardo has received funding from the following Research Projects/Grants: 
European Union's Horizon Europe research and innovation programme under the Marie Sklodowska-Curie Action MSCA-DN-101119556 (IN-DEEP). 
PID2023-146678OB-I00 funded by MICIU/AEI /10.13039/501100011033 and by FEDER, EU;
BCAM Severo Ochoa accreditation of excellence CEX2021-001142-S funded by MICIU / AEI / 10.13039/501100011033; 
Basque Government through the BERC 2022-2025 program;
BEREZ-IA (KK-2023/00012) and RUL-ET(KK-2024/00086), funded by the Basque Government through ELKARTEK;
Consolidated Research Group MATHMODE (IT1866-26) of the UPV/EHU given by the Department of Education of the Basque Government; 
BCAM-IKUR-UPV/EHU, funded by the Basque Government IKUR Strategy and by the European Union NextGenerationEU/PRTR. 
The research of Kristoffer van der Zee was supported by the Engineering and Physical Sciences Research Council (EPSRC), UK, under Grant EP/W010011/1

\paragraph{Declaration of generative AI and AI-assisted technologies in the manuscript preparation process}
During the preparation of this work the author(s) used Gemini in order to improve the English language, refine the phrasing, and format specific sections of the manuscript. After using this tool/service, the author(s) reviewed and edited the content as needed and take(s) full responsibility for the content of the published article.

\paragraph{CRediT contribution statement}
\textbf{Pablo Herrera:} Formal analysis, Investigation, Methodology, Software, Visualization, Writing - original draft, Writing - review \& editing. \textbf{Jamie M. Taylor:} Formal analysis, Methodology, Writing - review \& editing. \textbf{Carlos Uriarte:} Conceptualization, Formal analysis, Methodology, Writing - review \& editing. \textbf{Ignacio Muga:} Conceptualization, Formal analysis, Writing - review \& editing. \textbf{David Pardo:} Funding acquisition, Supervision, Validation, Writing - review \& editing. \textbf{Kristoffer G. van der Zee:} Conceptualization, Formal analysis, Methodology, Writing - review \& editing.

\bibliographystyle{siam}
\bibliography{bibliography}
\pagebreak
\appendix

\section{Proof of main theorems}
\subsection{Proof of Theorem~\ref{thm:approach1_inexact}}
\label{sec:proof_thm1}
Since $\rho<2\|B\|^{-2}$, by Proposition~\ref{prop:contraction} we have that $\|I-\rho B'B\|<1$. Thus, $1-\|I-\rho B'B\|$ is positive.

Let us define $e^k= u^k-u^*$. First, observe that $r^k=B(u^*-u^k)=-Be^k$. Hence, using~\eqref{eq:relative1} and triangular inequality, we have:
\begin{equation}\label{eq:rk}
\|r^k_\varepsilon\|_\mathbb{V} \le 
\|r^k_\varepsilon-r^k\|_\mathbb{V}+\|r^k\|_\mathbb{V}\leq (\varepsilon+1)\|r^k\|_\mathbb{V}\leq (\varepsilon+1)\|B\| \|e^k\|_\mathbb U.
\end{equation}
Moreover,

\begin{equation}\label{eq:dk}
\|\delta_\varepsilon^k\|_\mathbb{U}\leq \|\delta_\varepsilon^k-\delta^k\|_\mathbb{U}+\|\delta^k\|_\mathbb{U}
\leq (\varepsilon+1)\|\delta^k\|_\mathbb{U}=(\varepsilon+1)\|B'r_\varepsilon^k\|_\mathbb{U}
\leq (1+\varepsilon)^2\|B\|^2 \|e^k\|_\mathbb{U}. 
\end{equation}

Thus, it is enough to prove that $\{e^k\}$ is a contraction. Notice that
$$
\begin{array}{rl}
e^{k+1}= & u^{k+1}-u^*\\
= &  u^k+\rho\delta_\varepsilon^k -u^*\\
= & e^k + \rho(\delta^k_\varepsilon-\delta^k) +\rho \delta^k\\
= & e^k + \rho(\delta^k_\varepsilon-\delta^k) + \rho B'r_\varepsilon^k\\
= & e^k + \rho(\delta^k_\varepsilon-\delta^k) + \rho B'(r_\varepsilon^k-r^k) + \rho B'r^k\\
= & (I-\rho B'B)e^k + \rho(\delta^k_\varepsilon-\delta^k) + \rho B'(r_\varepsilon^k-r^k).
\end{array}
$$
Therefore,
$$
\begin{array}{rl}
\|e^{k+1}\|_\mathbb{U} \leq &  \|I-\rho B'B\| \|e^k\|_\mathbb{U}
+\rho \varepsilon( \|\delta^k\|_\mathbb{U} + \|B\|\|r^k\|_\mathbb{V})\\
\le &  \left(\|I-\rho B'B\|+ \rho (\varepsilon^2+2\varepsilon)\|B\|^2\right)\|e^k\|_\mathbb{U}\,.\\
\end{array}
$$
By the hypothesis of the Theorem, the constant $\|I-\rho B'B\|+ \rho (\varepsilon^2+2\varepsilon)\|B\|^2$ is less than $1$, which proves that $\{e^k\}$ is a contraction.

\subsection{Proof of Theorem~\ref{thm:approach2_inexact}}
\label{sec:proof_thm2}

Let us define $e^k_\varepsilon= u^k_\varepsilon-u^*$ and observe that $r^k=B(u^*-u^k_\varepsilon)=-Be^k_\varepsilon$. Analogously to~\eqref{eq:rk} and~\eqref{eq:dk}, we have:
$$
\|r^k_\varepsilon\|_\mathbb{V} \leq (\varepsilon+1)\|B\| \|e^k_\varepsilon\|_\mathbb U
\qquad\hbox{ and }\qquad 
\|\delta_\varepsilon^k\|_\mathbb{U}\leq (1+\varepsilon)^2\|B\|^2 \|e^k_\varepsilon\|_\mathbb{U}. 
$$
Thus, it remains to prove that $\{e^k_\varepsilon\}$ is a contraction.
Notice that:
$$
\begin{array}{rl}
e^{k+1}_\varepsilon= & u^{k+1}_\varepsilon-u^*\\
= &  u^{k+1}_\varepsilon-u^{k+1}+u^k_\varepsilon+\rho\delta_\varepsilon^k -u^*\\
= & u^{k+1}_\varepsilon-u^{k+1} + e^k_\varepsilon + \rho(\delta^k_\varepsilon-\delta^k) +\rho \delta^k\\
= & u^{k+1}_\varepsilon-u^{k+1} + e^k_\varepsilon + \rho(\delta^k_\varepsilon-\delta^k) + \rho B'r_\varepsilon^k\\
= & u^{k+1}_\varepsilon-u^{k+1} + e^k_\varepsilon + \rho(\delta^k_\varepsilon-\delta^k) + \rho B'(r_\varepsilon^k-r^k) + \rho B'r^k\\
= & u^{k+1}_\varepsilon-u^{k+1} + (I-\rho B'B)e^k_\varepsilon + \rho(\delta^k_\varepsilon-\delta^k) + \rho B'(r_\varepsilon^k-r^k).
\end{array}
$$
Therefore,
$$
\begin{array}{rl}
\|e^{k+1}_\varepsilon\|_\mathbb{U} \leq & \rho\varepsilon\|\delta_\varepsilon^k\|_\mathbb{U} + \|I-\rho B'B\| \|e^k_\varepsilon\|_\mathbb{U}
+\rho \varepsilon( \|\delta^k\|_\mathbb{U} + \|B\|\|r^k\|_\mathbb{V})\\
\le &  \left(\|I-\rho B'B\|+ \rho (\varepsilon^3+3\varepsilon^2+3\varepsilon)\|B\|^2\right)\|e^k_\varepsilon\|_\mathbb{U}\,.\\
\end{array}
$$
By the hypothesis of the Theorem, the constant $\|I-\rho B'B\|+ \rho (\varepsilon^3+3\varepsilon^2+3\varepsilon)\|B\|^2$ is less than $1$, which proves that $\{e^k_\varepsilon\}$ is a contraction.

\end{document}